\documentclass[final,leqno]{siamltex}

\usepackage{graphicx} 
\usepackage{epsfig} 
\usepackage{times} 
\usepackage{amsmath} 
\usepackage{amssymb}  
\usepackage{amsfonts}
\usepackage{epstopdf}
\usepackage{subfig}
\usepackage[sort]{cite}
\usepackage{comment}
\usepackage{mathtools}
\usepackage{eurosym}
\usepackage{multirow}
\usepackage{tikz}
\usetikzlibrary{arrows}

\newtheorem{remark}[theorem]{Remark}
\newtheorem{assumption}[theorem]{Assumption}
\newtheorem{problem}[theorem]{Problem}
\newtheorem{algorithm}[theorem]{Algorithm}
\newcommand{\sr}{\stackrel}

\newcommand{\rar}{\rightarrow}

\newcommand{\be}{\begin{equation}}
\newcommand{\ee}{\end{equation}}
\newcommand{\bea}{\begin{eqnarray}}
\newcommand{\eea}{\end{eqnarray}}
\newcommand{\bes}{\begin{eqnarray*}}
\newcommand{\ees}{\end{eqnarray*}}
\newcommand{\bi}{\begin{itemize}}
\newcommand{\ei}{\end{itemize}}
\newcommand{\ben}{\begin{enumerate}}
\newcommand{\een}{\end{enumerate}}
\newcommand{\bp}{\begin{problem}}
\newcommand{\ep}{\end{problem}}
\newcommand{\hso}{\hspace{.1in}}
\newcommand{\hst}{\hspace{.2in}}

\newcommand{\hse}{\hspace{.05in}}

\title{Dynamic Programming Subject to Total Variation Distance Ambiguity}

\author{Ioannis~Tzortzis\thanks{Department of Electrical and Computer Engineering, University of Cyprus (UCY), Nicosia, Cyprus. ({\tt tzortzis.ioannis@ucy.ac.cy}).}
        \and Charalambos~D.~Charalambous\thanks{Department of Electrical and Computer Engineering, University of Cyprus (UCY), Nicosia, Cyprus. ({\tt chadcha@ucy.ac.cy}).}
\and Themistoklis~Charalambous\thanks{School of Electrical Engineering, Royal Institute of Technology (KTH), Stockholm, Sweden. ({\tt themisc@kth.se}).}}

\begin{document}
\maketitle

\begin{abstract}
The aim of this paper is to address optimality of stochastic control strategies via dynamic programming subject to total variation distance ambiguity on the conditional distribution of the controlled process. We formulate the stochastic control problem using minimax theory, in which the control minimizes the pay-off while the conditional distribution, from the total variation distance set, maximizes it.
 
First, we investigate the maximization of a linear functional on the space of probability measures on abstract spaces, among those probability measures which are within  a total variation distance from a nominal probability measure, and then we give the maximizing probability measure in closed form. Second, we utilize the solution of the maximization to solve minimax stochastic control with deterministic control strategies, under a Markovian and a non-Markovian assumption, on the conditional distributions of the controlled process. The results of this part include: 1) Minimax optimization subject to total variation distance ambiguity constraint; 2) new dynamic programming recursions, which involve the oscillator seminorm of the value function, in addition to the standard terms; 3) new infinite horizon discounted dynamic programming equation, the associated contractive property, and a new policy iteration algorithm. Finally, we provide illustrative examples for both the finite and infinite horizon cases. For the infinite horizon case we invoke the new policy iteration algorithm to compute the optimal strategies.
\end{abstract}



\pagestyle{myheadings}
\thispagestyle{plain}
\markboth{I. TZORTZIS, C. D. CHARALAMBOUS AND T. CHARALAMBOUS}{DYNAMIC PROGRAMMING SUBJECT TO TOTAL VARIATION DISTANCE AMBIGUITY}

\section{Introduction}
Dynamic programming recursions are often employed in optimal control and decision theory to establish existence of optimal strategies, to derive necessary and sufficient optimality conditions, and to compute the optimal strategies either in closed form or via algorithms \cite{caines88,varayia86,vanSchuppen10}. The cost-to-go and the corresponding dynamic programming recursion, in their general form, are functionals of the conditional distribution of the underlying state process (controlled process) given the past and present state and control processes \cite{caines88}. Thus, any ambiguity of the controlled process conditional distribution will affect the optimality of the strategies. The term ``ambiguity'' is used to differentiate from the term ``uncertainty" often used in control nomenclature to account for situations in which the true and nominal distribution (induced by models) are absolutely continuous, and hence they are defined on the same state space. This distinction is often omitted from various robust deterministic and stochastic control approaches, including minimax and risk-sensitive formulations \cite{ahmed99,baras,basar1995h,Bensoussan95,charhibey96,Charalambous07,elliotaggoun,James94,kumar81,Petersen,pra96,Ugrinovskii,whittle}. In this paper, the class of models is described by a ball with respect to the total variation distance between the nominal distribution and the true distribution, hence it admits distributions which are singular  with respect to the nominal distribution.

The main objective of this paper is to investigate the effect on the cost-to-go and dynamic programming of the ambiguity in the controlled process conditional distribution, and hence on the optimal decision strategies. Specifically, we quantify the conditional distribution ambiguity of the controlled process by a ball with respect to the total variation distance metric, centered at a nominal conditional distribution, and then we derive a new dynamic programming using minimax theory, with two players: player I the control process and player II the conditional distribution (controlled process), opposing each other actions. In this minimax game formulation, player's I objective is to  minimize the cost-to-go, while player's II objective is to maximize it. The maximization over the total variation distance ball of player II is addressed by first deriving results related to the maximization of linear functionals on a subset of the space of signed measures. Utilizing these results, a new dynamic programming recursion is presented which, in addition to the standard terms, includes additional terms that codify the level of ambiguity allowed by player II with respect to the total variation distance ball. Thus, the effect of player I, the control process, is to minimize, in addition to the classical terms, the difference between the maximum and minimum values of the cost-to-go, scaled by the radius of the total variation distance ambiguity set. We treat in a unified way the finite horizon case, under both the Markovian and non-Markovian nominal controlled processes, and the infinite horizon case. For the infinite horizon case we consider a discounted pay-off and we show that the operator associated with the resulting dynamic programming equation under total variation distance ambiguity is contractive. Consequently, we derive a new policy iteration algorithm to compute the optimal strategies. Finally, we provide examples for the finite and for the infinite horizon case.

Previous related work on optimization of stochastic systems subject to total variation distance ambiguity is found in  \cite{fcn2012} for continuous time controlled diffusion processes described by It\^{o} differential equations. However, the solution method employed in \cite{fcn2012} is fundamentally different; it approaches the maximization problem indirectly, by employing Large Deviations concepts to derive the maximizing measure as a convex combination of a tilted probability measure and the nominal measure, under restrictions on the class of measures considered. The dynamic programming equation derived in \cite{fcn2012} is limited by the assumption that the maximizing measure is absolutely continuous with respect to the nominal measure.

In this paper, our focus is to understand the effect of total variation distance ambiguity of the conditional distribution on dynamic programming, from a different point of view, utilizing concepts from signed measures. Consequently, we derive a new dynamic programming recursion which depends explicitly on the radius of the total variation distance, the closed form expression of the maximizing measure, or the oscillator seminorm of the value function. One of the fundamental properties of the maximizing conditional distribution is that, as the ambiguity radius increases, the maximizing conditional distribution becomes singular with respect to the nominal distribution. The point to be made here is that the total variation distance ambiguity set admits controlled process distributions which are not necessarily defined on the same state space as the nominal controlled process distribution. In terms of robustness of the optimal policies, this additional feature is very attractive compared to minimax techniques based on relative entropy uncertainty or risk-sensitive pay-offs \cite{ahmed99,baras,basar1995h,Bensoussan95,charhibey96,Charalambous07,elliotaggoun,James94,kumar81,Petersen,pra96,Ugrinovskii,whittle}, because often the true controlled distribution lies on a higher dimensional state space compared to the nominal controlled process distribution. 

The rest of the paper is organized as follows. In Section \ref{idp}, we give  a high level discussion on  classical dynamic programming for MCM and we present  some aspects of the problems and results obtained in the paper. In Section \ref{abs}, we describe the abstract formulation of the minimax problem under total variation distance ambiguity, and we derive the closed form expression of the maximizing measure. In Section~\ref{partially}, we apply the
abstract setup to Feedback Control Model (FCM) (e.g., non-Markov) and to MCM. We derive new dynamic programming recursions which characterize the optimality of minimax strategies. In Section~\ref{sec.dpihc}, we treat the infinite horizon case, where we show that the dynamic programming operator is contractive, and we develop a new policy iteration algorithm. Finally, in Section~\ref{sec.examples} we present various examples to illustrate the applications of the new dynamic programming recursions.

\subsection{Discussion on the Main Results}
\label{idp}
Next, we describe at a high level the results obtained in this paper.

\subsubsection{Dynamic Programming of Finite Horizon Discounted-Markov Control Model}
A finite horizon Discounted-Markov Control Model (D-MCM) with deterministic strategies is a septuple
\begin{equation}
\begin{multlined}\label{mcm}
\mbox{D-MCM}:\Big(\{ {\cal X}_i\}_{i=0}^n, \{ {\cal U}_i\}_{i=0}^{n-1}, \{ {\cal U}_i(x_i): x_i \in {\cal X}_i\}_{i=0}^{n-1}, \{ Q_i(dx_i| x_{i-1}, u_{i-1}):\\
  (x_{i-1}, u_{i-1}) \in {\cal X}_{i-1} \times {\cal U}_{i-1} \}_{i=0}^{n},\{f_i\}_{i=0}^{n-1}, h_n,\alpha\Big)
\end{multlined}\end{equation} consisting of

(a) State Space. A sequence of Polish spaces (complete separable metric spaces) $\{ {\cal X}_i: i=0, \ldots, n\}$, which model the state space of the controlled random process $\{x_j\in {\cal X}_j: j=0, \ldots, n\}$.

(b) Control or Action Space. A sequence of Polish spaces $\{ {\cal U}_i: i=0, \ldots, n-1\}$, which model the control or action set of the control random process $\{u_j\in {\cal U}_j: j=0, \ldots, n-1\}$.

(c) Feasible Controls or Actions. A family $\{ {\cal U}_i(x_i): x_i \in {\cal X}_i\}$ of non-empty measurable subsets ${\cal U}_i(x_i)$ of ${\cal U}_i$, where ${\cal U}_i(x_i)$ denotes the set of feasible controls or actions, when the controlled process is in state $x_i \in {\cal X}_i$, and  the feasible state-actions pairs defined by ${\mathbb K}_i \triangleq \Big\{ (x_i, u_i): x_i \in {\cal X}_i, u_i \in {\cal U}_i(x_i)\Big\} $ are measurable subsets of  ${\cal X}_i \times {\cal U}_i, i=0, \ldots, n-1$.

(d) Controlled Process Distribution. A collection of conditional distributions or stochastic kernels $Q_i(dx_i|x_{i-1}, u_{i-1})$ on ${\cal X}_i$
given $(x_{i-1}, u_{i-1}) \in {\mathbb K}_{i-1} \subseteq {\cal X}_{i-1} \times {\cal U}_{i-1}, i=0, \ldots, n$. The controlled process distribution  is described by the sequence of transition probability distributions $\{Q_i(dx_i|x_{i-1}, u_{i-1}): (x_{i-1},u_{i-1})\in {\mathbb K}_{i-1}, i=0, \ldots, n\}$.

(e) Cost-Per-Stage. A collection of non-negative measurable functions $f_j: {\mathbb K}_j \rar [0, \infty]$, called the cost-per-stage, such that $f_j(x,\cdot)$ does not take the value $+\infty$ for each $x \in {\cal X}_j, j=0, \ldots, n-1$. The running pay-off functional is defined in terms of $\{f_j:j=0,\hdots,n-1\}$.

(f) Terminal Cost. A bounded measurable non-negative function $h_n: {\cal X}_n \rar [0, \infty)$ called the terminal cost. The pay-off functional at the last stage is defined in terms of $h_n$.

(g) Discounting Factor. A real number $\alpha\in (0,1)$ called the discounting factor.

The definition of D-MCM envisions applications of systems described by discrete-time dynamical state space models, which include random external inputs, since such models give rise to a collection of controlled processes distributions $\{Q_i(dx_i|x_{i-1},u_{i-1}){:}(x_{i-1},u_{i-1})\in {\mathbb K}_{i-1}, i=0,\hdots,n\}$. For any integer $j\geq 0$, define the product spaces by ${\cal X}_{0,j}\triangleq\times_{i=0}^j{\cal X}_i$ and ${\cal U}_{0,j-1}\triangleq\times_{i=0}^{j-1}{\cal U}_i$. Define the discounted sample pay-off by\begin{align} F^\alpha_{0,n}(x_0,u_0,x_1,u_1,\ldots, x_{n-1},u_{n-1},x_n)\triangleq\sum_{j=0}^{n-1} \alpha^jf_j(x_j, u_j) +\alpha^n h_n(x_n). 
\end{align}

The goal in Markov controlled optimization with deterministic strategies is to choose a control strategy or policy
$g \triangleq \{g_j:j=0,1,\ldots,n-1\}$, $g_j:{\cal X}_{0,j}\times {\cal U}_{0,j-1}\longrightarrow {\cal U}_j(x_j)$, $u_j^g=g_j(x_0^g,x_1^g,\ldots,x_j^g,u_0^g,u_1^g,\ldots,u_{j-1}^g)$,  $j=0, 1, \ldots, n-1$
so as to minimize the pay-off functional
\begin{align} 
\label{tc}&{\mathbb E} \Big\{ \sum_{j=0}^{n-1} \alpha^jf_j(x_j^g, u_j^g) +\alpha^nh_n(x_n^g) \Big\} = \int_{ {\cal X}_0 \times {\cal X}_1\times \ldots \times{\cal X}_n}\nonumber\\
&\quad F^\alpha_{0,n}\Big(x_0,u_0^g(x_0),x_1,u_1^g(x_0,x_1),\ldots,x_{n-1},u_{n-1}^g(x_0,x_1,\ldots,x_{n-1}),x_n)  \\ &\qquad Q_0(dx_0)Q_1(dx_1|x_{0}, u_{0}^g(x_0))\ldots Q_n(dx_n|x_{n-1}, u_{n-1}^g(x_0,x_1,\ldots, x_{n-1} ) \Big).\nonumber
\end{align}
Clearly, pay-off (\ref{tc}) is a functional of the collection of conditional distributions $\{Q_i(\cdot|\cdot): i=0,1, \ldots, n\}$. Moreover, if this collection of distribution has countable support for each $(x_{i-1},u_{i-1})$, $i=0,\hdots,n$, then each integral in (\ref{tc}) is reduced to a countable summation.

For $(i, x) \in \{0,1 \ldots,n \} \times {\cal X}_i$, let $V_i^0(x) \in {\mathbb R}$ represent the minimal cost-to-go or value function on the time horizon $\{i, i+1, \ldots, n\}$ if the controlled process starts at state $x_i=x$ at time $i$, defined by \bea
V_i^0(x) \triangleq \inf_{\sr{ g_k \in {\cal U}_{k}(x_k)}{ k=i, \ldots, n-1}} {\mathbb E}^g_{i,x} \Big\{ \sum_{j=i}^{n-1} \alpha^jf_j(x_j^g, u_j^g) +\alpha^nh_n(x_n^g) \Big\} \label{v}
\eea
where ${\mathbb E}^g_{i,x}\{\cdot\}$ denotes expectation conditioned on $x_i^g=x$. A Markov property on the controlled process distributions, i.e., $Q_i(dx_i|x^{i-1},u^{i-1}){=}Q_i(dx_i|x_{i-1},u_{i-1})$, $\forall(x^{i-1},u^{i-1})\in \times_{j=0}^{i-1}{\mathbb K}_j$, $i=0,1,\hdots,n$, under admissible non-Markov strategies, implies that Markov control strategies are optimal \cite{varayia86}. Consequently, it can be shown that the value function (\ref{v}) satisfies the following dynamic programming recursion relating the value functions $V_i^0(\cdot)$ and $V_{i+1}^0(\cdot)$ \cite{varayia86},
\begin{align}
V_n^0(x) &= \alpha^nh_n(x), \hst x \in {\cal X}_n \label{dp2}\\
V_i^0(x) & = \inf_{ u \in {\cal U}_i(x)} \Big\{ \alpha^if_i(x, u) + \int_{{\cal X}_{i+1}} V_{i+1}^0(z) Q_{i+1}(dz| x,u)\Big\}, \hst x \in {\cal X}_i. \label{dp1} 
\end{align}
Since the value function $V_i^0(x)$ defined by (\ref{v}) and the dynamic programming recursion (\ref{dp2}), ({\ref{dp1}) depend on the complete knowledge of the collection of conditional distributions $\{ Q_i(\cdot| \cdot): i=0, \ldots, n\}$, any mismatch of the collection $\{ Q_i(\cdot| \cdot): i=0, \ldots, n\}$ from the true collection of conditional distributions, will affect the optimality of the control strategies. Our objective is to address the impact of any ambiguity measured by the total variation distance between the true conditional distribution and a given nominal distribution on the cost-to-go (\ref{v}), and dynamic programming recursion (\ref{dp2}), (\ref{dp1}).

\subsubsection{Dynamic Programming of Infinite Horizon D-MCM}
The infinite horizon D-MCM with deterministic strategies is a special case of the finite horizon D-MCM specified by a six-tuple
\begin{align}\label{ihdmcm}
\Big( {\cal X},  {\cal U}, \{ {\cal U}(x): x\in {\cal X}\},\{ Q(dz| x, u): (x, u) \in {\cal X} \times {\cal U} \},f,\alpha\Big)
\end{align}
where the elements defined under (a)-(f) are independent of time index $i$. That is, the state space is ${\cal X}$, the control or action space is ${\cal U}$, the feasible controls or actions is a family $\{{\cal U}(x):x\in{\cal X}\}\subset {\cal U}$, the controlled process distribution is a stochastic kernel $Q(\cdot|\cdot)$ on ${\cal X}$ given ${\mathbb K}$, where ${\mathbb K} \triangleq \Big\{ (x, u): x\in{\cal X}, u \in {\cal U}(x)\Big\}$, the cost-per-stage is a one stage cost $f:{\mathbb K}\longrightarrow[0,\infty]$, and there is no terminal cost (it is set to zero).

The dynamic programming equation of the infinite horizon D-MCM as given by \cite{vanSchuppen10} is a function $v_{\infty}^0:{\cal X}\longrightarrow \mathbb{R}$ satisfying 
\begin{align} v_{\infty}^0(x)=\inf_{u\in{\cal U}(x)}\Big\{f(x,u)+\alpha\int_{{\cal X}}v_{\infty}^0(z)Q(dz|x,u)\Big\},\hst x\in{\cal X}. \label{dpihdmcm}\end{align}
Similarly to the finite horizon D-MCM, the dynamic programming equation (\ref{dpihdmcm}) depends on the conditional distribution $Q(dz|x,u)$, hence any ambiguity or mismatch of $Q(dz|x,u)$ from the true distribution affects optimality of the strategies.\\

\subsubsection{Dynamic Programming with Total Variation Distance Ambiguity}
 Motivated by the above discussion, the objective of this paper is to investigate dynamic programming under ambiguity of the conditional distributions of the controlled processes $$\Big\{ Q_i(dx_i| x_{i-1}, u_{i-1}): (x_{i-1}, u_{i-1}) \in {\mathbb K}_{i-1}\Big\},\hso i=0,\ldots, n.$$ The ambiguity of the conditional distributions of the controlled process is modeled by the total variation distance. Specifically, given a collection of nominal controlled process distributions
$\{ Q^o_i(dx_i| x_{i-1}, u_{i-1}): (x_{i-1}, u_{i-1}) \in {\mathbb K}_{i-1}\}$,  $i=0,\ldots, n$, the corresponding collection of true controlled process distributions $\{ Q_i(dx_i| x_{i-1}, u_{i-1}): (x_{i-1}, u_{i-1})\in {\mathbb K}_{i-1}\}$, $i=0, \ldots, n$, is modeled by a set described by the total variation distance centered at the nominal conditional distribution having radius $R_i \in [0,2]$, $i=0,\hdots,n$, defined by
\begin{equation}
\textbf{B}_{R_i}(Q^{o}_{{i}})  (x_{i-1},u_{i-1}) {\triangleq} \Big\{ { Q}_{{i}}(\cdot| x_{i-1},u_{i-1}){:}
|| { Q}_{{i}}(\cdot| x_{i-1},u_{i-1}) {-}Q^{o}_{{i}}(\cdot| x_{i-1},u_{i-1}) ||_{TV}  {\leq} R_i \Big\}.\nonumber
 \end{equation}
Here $|| \cdot ||_{TV}$ denotes the total variation distance between two probability measures, $|| \cdot ||_{TV}: {\cal M}_1(\Sigma) \times {\cal M}_1(\Sigma)\longmapsto [0,\infty]$ defined by 
\begin{equation}
||\alpha-\beta ||_{TV} \triangleq \sup_{ P \in {\cal P}(\Sigma)} \sum_{F_i \in P} |\alpha (F_i)- \beta(F_i)|,\hso \alpha, \beta\in {\cal M}_1(\Sigma)\label{totalv}
\end{equation}
where ${\cal M}_1(\Sigma)$ denotes the set of probability measures on ${\cal B}(\Sigma)$ and ${\cal P}(\Sigma)$ denotes the collection of all finite partitions of $\Sigma$. Note that the distance metric (\ref{totalv}) induced by the total variation norm does  not require absolute continuity of the measures $\alpha\in{\cal M}_1(\Sigma)$ and $\beta\in{\cal M}_1(\Sigma)$. The total variation distance model of ambiguity is quite general, and it includes linear, non-linear, finite and/or countable state space models, etc, since no assumptions are impossed on the structure of the stochastic control dynamical system model, which induces the collection of conditional distributions $\{Q_i(\cdot|\cdot):i=0,\ldots,n\}$, $\{Q^o_i(\cdot|\cdot):i=0,\ldots,n\}$. Given the above description of ambiguity in distribution, we re-formulate the value function and dynamic programming recursion via minimax theory as follows.

 For $(i, x) \in \{0,1 \ldots, n \} \times {\cal X}_i$, let $V_i(x) \in {\mathbb R}$ represent the minimal cost-to-go on the time horizon $\{i, i+1, \ldots, n\}$ if the state of the controlled process starts at state $x_i=x$ at time $i$, defined by
\begin{align*}
V_i(x) \triangleq \inf_{  \sr{g_k \in {\cal U}_{k}(x_k)}{ k=i, \ldots, n-1     }}  \sup_{ \sr{Q_{k+1}(\cdot| x_{k}, u_{k}) \in \textbf{B}_{R_{k+1}}(Q^o_{k+1})(x_k,u_k)}{k=i, \ldots, n-1}}    {\mathbb E}^g_{i,x} \Big\{ \sum_{j=i}^{n-1} \alpha^jf_j(x_j^g, u_j^g) +\alpha^nh_n(x_n^g) \Big\} 
\end{align*}
 where ${\mathbb E}^g_{i,x}$ denotes conditional expectation with respect to the true collection of conditional distribution $\{Q_k(\cdot|\cdot):k=i,\hdots,n\}$. Even in the above minimax setting the Markov property of the controlled process distribution under an admissible non-Markov strategy implies that Markov control strategies are optimal. Moreover, the value function satisfies the following dynamic programming recursion relating the value function $V_i(\cdot)$ and $V_{i+1}(\cdot)$, for all $i=0,1,\hdots,n-1$.
\begin{align*}
&V_n(x) =\alpha^nh_n(x), \quad\mbox{$x \in {\cal X}_n$} \\
& V_i(x)  = \inf_{\mathclap{ u \in {\cal U}_i(x)}} \ \ \ \sup_{ Q_{i+1}(\cdot| x_, u) \in \textbf{B}_{R_{i+1}}(Q^o_{{i+1}})(x,u)} \Big\{ \alpha^if_i(x, u) {+} \int_{\mathrlap{{\cal X}_{i+1}}} V_{i+1}(z) Q_{i+1}(dz| x,u)\Big\}, \ \mbox{$ x \in {\cal X}_i $}.
\end{align*}
Based on this formulation, if $V_{i+1}(\cdot)$ is bounded continuous non-negative, we show that the new dynamic programming equation is given by
\begin{eqnarray}
  \label{ndp2} V_n(x) &=&\alpha^nh_n(x), \quad\mbox{$ x \in {\cal X}_n$}\\
  V_i( x) &=&\inf_{u \in {\cal U}_i(x)}\Big\{\alpha^i f_i(x,u)+ \int_{{\cal X}_{i+1}}  V_{i+1}(z) Q^o_{i+1}(dz| x,u) \nonumber\\ [-1.5ex]\label{ndp1}\\[-1.5ex]
  &&+ \frac{R_i}{2} \Big( \sup_{z \in{\cal X}_{i+1}} V_{i+1}(z)- \inf_{z \in{\cal X}_{i+1}} V_{i+1}(z)  \Big)   \Big\},  \quad\mbox{$ x \in {\cal X}_i$}. \nonumber
\end{eqnarray}
Note that the new term in the right side of (\ref{ndp1}) is the oscillator seminorm of $V_{j+1}(\cdot)$ called the global modulus of continuity of $V_{j+1}(\cdot)$, which measures the difference between the maximum and minimum values of $V_{j+1}(\cdot)$.

For the infinite horizon D-MCM the new dynamic programming equation is given by
\begin{equation} \begin{multlined}
\label{ihDMCM}v_{\infty}(x)=\inf_{u\in{\cal U}(x)}\Big\{f(x,u)\\
+\alpha\int_{{\cal X}}v_{\infty}(z)Q^o(dz|x,u)
+\alpha\frac{R}{2}\Big (\sup_{z\in{\cal X}}v_{\infty}(z)-\inf_{z\in{\cal X}}v_{\infty}(z)\Big)\Big\},\quad\mbox{$ x\in{\cal X}$}.\end{multlined}
\end{equation}
For finite and countable alphabet spaces ${\cal X}_j$, ${\cal X}$, the integrals in the right hand side of (\ref{ndp1}), (\ref{ihDMCM}) are replaced by summations.
 
 In addition to the D-MCM, we will also discuss the general discounted feedback control model (i.e., we relax the Markovian assumption). In summary, the issues discussed and results obtained in this paper are the following:
(1) formulation of finite  horizon discounted stochastic optimal control  subject to conditional distribution ambiguity described by total variation distance via minimax theory;
(2) dynamic programming recursions for a) nominal D-MCM, and b) Discounted-Feedback Control Model (D-FCM), under total variation distance ambiguity on the conditional distribution of the controlled process;
(3) formulation of the infinite horizon D-MCM and dynamic programming equation under conditional distribution ambiguity described by total variation distance via minimax theory;
(4) characterization of the maximizing conditional distribution belonging to the total variation distance set, and the corresponding new dynamic programming recursions;
(5) contraction property of the infinite horizon D-MCM dynamic programming and new policy iteration algorithm;
(6) examples for the finite and infinite horizon cases.

\section{Maximization With Total Variation Distance Ambiguity}\label{abs}
In this section, we recall certain results from \cite{fcn2012} on the maximization of a linear functional on the space of probability distributions subject to total variation distance ambiguity. We use these results to derive the maximizing probability distribution subject to total variation distance ambiguity of the controlled process.

Let $(\Sigma,d_\Sigma)$ denote a complete, separable metric space (a Polish space), and $(\Sigma, {\cal B}(\Sigma))$ the corresponding measurable space, in which ${\cal B}(\Sigma)$ is the $\sigma$-algebra generated by open sets in $\Sigma$. Let ${\cal M}_1 (\Sigma)$ denote space of countably additive probability measures  on $(\Sigma, {\cal B}(\Sigma))$.  Define the spaces
\begin{align*}
&BC(\Sigma)\triangleq\big\{\mbox{Bounded continuous functions}\hse\ell:\Sigma\longrightarrow{\mathbb R}: ||\ell||\triangleq\sup_{x\in\Sigma}|\ell(x)|<\infty\big\}\\
&BM(\Sigma)\triangleq\big\{\mbox{Bounded measurable functions}\hse\ell:\Sigma\longrightarrow{\mathbb R}: ||\ell||<\infty\big\}\\
&C(\Sigma)\triangleq \big\{\mbox{Continuous functions}\hse\ell:\Sigma\longrightarrow{\mathbb R}: ||\ell||<\infty\big\},\quad C^+(\Sigma)\triangleq\big\{\ell\in C(\Sigma):\ell\geq 0\big\}\\
&BC^+(\Sigma)\triangleq\big\{\ell\in BC(\Sigma):\ell\geq 0\Big\},\hso BM^+(\Sigma)\triangleq\big\{\ell\in BM(\Sigma):\ell\geq 0\big\}.
\end{align*}
Clearly, $BC(\Sigma)$, $BM(\Sigma)$, $C(\Sigma)$ are Banach spaces. We present the maximizing measure for $\ell{\in} BC^+(\Sigma)$, although the results can be generalized to real-valued functions $\ell\in L^{\infty,+}(\Sigma, {\cal B}(\Sigma), \nu)$, the set of all ${\cal B}(\Sigma)$-measurable, non-negative  essentially bounded functions defined $\nu{-}a.e.$ endowed with the essential supremum norm $||\ell||_{\infty,\nu}{=}\nu \mbox{-} \mbox{ess }\sup_{x {\in} \Sigma} \ell(x)$.

From \cite{fcn2012}, we have the following. For $\ell\in BC^+(\Sigma)$, and $\mu\in {\cal M}_1(\Sigma)$ fixed, then \begin{equation}
 L(\nu^*)\triangleq \sup_{||\nu-\mu||_{TV}\leq R}\int_{\Sigma}\ell(x)\nu(dx)=\frac{R}{2}  \Big\{ \sup_{x \in {\Sigma}} \ell(x)  -  \inf_{x \in \Sigma} \ell(x) \Big\}    + \int_{\Sigma}\ell(x)\mu(dx)\label{f2n}
\end{equation} 
where $R\in[0,2]$, ${ \nu}^*$ satisfies the constraint $||{ \xi}^*||_{TV}= ||{ \nu}^*-{ \mu}||_{TV} =R$, it is normalized  ${ \nu}^*(\Sigma)=1$, and $\nu^*(A) \in [0,1]$ on any $A \in {\cal B}(\Sigma)$. Moreover, by defining \footnote{We adopt the standard definitions; infimum (supremum) of an empty set to be $+\infty$ ($-\infty$).}
\begin{align*}
&x^0 \in \Sigma^0  \triangleq  \{ x \in \Sigma: \ell(x) = \sup \{\ell(x): x\in \Sigma\} \equiv \ell_{\max}\} \\
&x_0 \in \Sigma_0  \triangleq \{ x \in \Sigma: \ell(x) = \inf\{\ell(x): x\in \Sigma\} \equiv \ell_{\min} \}
\end{align*}
then, the pay-off ${ L}({ \nu}^*)$ can be written as
\begin{align}
{ L}( \nu^*)&= \int_{\Sigma^0}     \ell_{\max}   \nu^*(dx) +   \int_{\Sigma_0}   \ell_{\min}     \nu^*(dx) +   \int_{\Sigma \setminus \Sigma^0\cup \Sigma_0} \ell(x) \mu(dx)  \label{n2}
\end{align}
and the optimal distribution ${\nu}^*\in {\cal M}_1 (\Sigma)$, which satisfy the total variation constraint, is given by
\begin{align}
\int_{\Sigma^0} \nu^*(dx)= \mu(\Sigma^0) + \frac{R}{2}\in[0,1]&, \qquad
 \int_{\Sigma_0} \nu^*(dx)= \mu(\Sigma_0) - \frac{R}{2}\in[0,1] \nonumber\\[-1.5ex]\label{cond1} \\[-1.5ex]
 \nu^*(A)= \mu(A)&, \hso \forall A \subseteq \Sigma \setminus \Sigma^0\cup \Sigma_0. \nonumber
\end{align}
Note that if $\Sigma^0=\Sigma_0=\{\emptyset\}$ then $\nu(\Sigma^0)=\nu(\Sigma_0)=0$, and $L(\nu^*)=\int_{\Sigma\setminus\Sigma^0\cup\Sigma_0}\ell(x)\mu(dx)$.

The second right hand side term in (\ref{f2n}) is related to the oscillator semi-norm of $f\in BM(\Sigma)$, called the global modulus of continuity, and it is defined by
\begin{align} \mbox{osc}(f)  & \triangleq \sup_{(x,y)\in\Sigma\times\Sigma}|f(x)-f(y)| =2\inf_{\beta\in \mathbb{R}}||f-\beta||,\hst\mbox{for}\hse f\in BM(\Sigma).\nonumber\end{align}
However, for $f\in BM^+(\Sigma)$ then \begin{align} \mbox{osc}(f)=\sup_{x\in\Sigma}|f(x)|-\inf_{x\in\Sigma}|f(x)|=\sup_{x\in\Sigma}f(x)-\inf_{x\in\Sigma}f(x). \nonumber\end{align}
Note that the above results can be extended to $f\in C^+(\Sigma)$.\\

{\it The Maximizing Measure for Finite and Countable Alphabet Spaces}

Here, we further elaborate on the form of the maximizing measures for finite and countable alphabet spaces, since we use them to analyze finite horizon D-MCM and D-FCM, and infinite horizon D-MCM with finite (or countable) state and control spaces.

Let $\Sigma$ be a non-empty denumerable set endowed with the discrete topology including finite cardinality $|\Sigma|$, with ${\cal M}_1(\Sigma)$ identified with the standard probability simplex in $\mathbb{R}^{|\Sigma|}$. That is, the set of all $|\Sigma|$-dimensional vectors which are probability vectors, $\{\nu(x):x\in\Sigma\}\in {\cal M}_1(\Sigma)$, $\{\mu(x):x\in\Sigma\}\in{\cal M}_1(\Sigma)$, and let $\ell \triangleq\{\ell(x):x\in\Sigma\}\in \mathbb{R}_+^{|\Sigma|}$. Define the maximum and minimum values of $\{\ell(x):x\in\Sigma\}$ by  
\begin{equation}
\ell_{\max}\triangleq \max_{x\in\Sigma}\ell(x),\quad \ell_{\min}\triangleq \min_{x\in\Sigma}\ell(x)\nonumber\end{equation}
 and its corresponding support sets by
\begin{equation}
 \Sigma^0 \triangleq \big\{x\in \Sigma:\ell(x)=\ell_{\max} \big\},\quad \Sigma_0\triangleq\big\{x\in \Sigma:\ell(x)=\ell_{\min} \big\}.\nonumber
\end{equation}
For all remaining sequence, $\big\{\ell(x):x\in \Sigma \setminus \Sigma^0\cup\Sigma_0\big\}$, and for $1\leq r\leq |\Sigma\setminus \Sigma^0\cup \Sigma_0|$, define recursively the set of indices for which the sequence achieves its $(k+1)^{th}$ smallest value by \begin{equation} \Sigma_k \triangleq\Big\{x\in \Sigma:\ell(x)=\min\big\{\ell(\alpha): \alpha \in \Sigma\setminus \Sigma^0\cup(\bigcup_{j=1}^k\Sigma_{j-1})\big\} \Big\},\hso k\in\{1,2,\hdots,r\}\nonumber\end{equation}
 till all the elements of $\Sigma$ are exhausted. Further, define the corresponding values of the sequence on sets $\Sigma_k$ by
 \bes \ell(\Sigma_k)\triangleq\min_{x\in\Sigma\setminus\Sigma^0\cup(\bigcup_{j=1}^k\Sigma_{j-1})}\ell(x),\hst k\in\{1,2,\hdots,r\}\ees
 where $r$ is the number of $\Sigma_k$ sets which is at most $|\Sigma\setminus\Sigma^0\cup\Sigma_0|$. For example, when $k=1$, $\ell(\Sigma_1)=\min_{x\in\Sigma\setminus\Sigma^0\cup\Sigma_0}\ell(x)$, when $k=2$, $\ell(\Sigma_2)=\min_{x\in\Sigma\setminus\Sigma^0\cup\Sigma_0\cup\Sigma_1}\ell(x)$ and so on.

 In \cite{ctlthem2013} it is shown that the maximum pay-off subject to total variation constraint is given by \begin{align} { L}( { \nu}^*)=\ell_{\max}\nu^*(\Sigma^0)+\ell_{\min}\nu^*(\Sigma_0)+\sum_{k=1}^r\ell(\Sigma_k)\nu^*(\Sigma_k), \label{mpoff1}   \end{align}
and that the optimal probabilities are given by (a water-filling) the following equations.
\begin{eqnarray}\label{all3}
&&\nu^*(\Sigma^0)\triangleq \sum_{x\in\Sigma^0}\nu^*(x)=\sum_{x\in\Sigma^0}\mu(x)+\frac{\alpha}{2}\equiv \mu(\Sigma^0)+\frac{\alpha}{2}\label{all3a}\\
&&\nu^*(\Sigma_0)\triangleq \sum_{x\in\Sigma_0}\nu^*(x)=\Big(\sum_{x\in\Sigma_0}\mu(x)-\frac{\alpha}{2}\Big)^+\equiv \Big(\mu(\Sigma_0)-\frac{\alpha}{2}\Big)^+\label{all3b}\\
&&\nu^*(\Sigma_k)\triangleq \sum_{x\in\Sigma_k} \nu^*(x)=\Big(\sum_{x\in\Sigma_k} \mu(x)-\Big(\frac{\alpha}{2}-\sum_{j=1}^k\sum_{x\in\Sigma_{j-1}}\mu(x)\Big)^+\Big)^+\nonumber\\ [-1.5ex]\label{all3c}\\[-1.5ex]
&& \qquad \qquad \qquad \qquad \quad \ \equiv \Big(\mu(\Sigma_k)-\Big(\frac{\alpha}{2}-\sum_{j=1}^k\mu(\Sigma_{j-1})\Big)^+\Big)^+\nonumber\\
&&\alpha \triangleq \min(R,R_{\max}),\ \ R_{\max}\triangleq 2(1-\sum_{x\in\Sigma^0}\mu(x))\equiv 2(1-\mu(\Sigma^0)),\ \  \mbox{ $R\in[0,2]$}\label{all3d}
\end{eqnarray}
where $k\in\{1,2,\ldots,r\}$ and $r$ is the number of $\Sigma_k$ sets which is at most $|\Sigma\setminus\Sigma^0\cup\Sigma_0|$.

The parameter $\alpha$ reinforces the intuitive notion of the total variation between the true and nominal probability distribution as having attributes similar to ``physical mass". Thus, if $\alpha=R_{\max}$, then \eqref{all3a} implies that the probability ``mass" on $\Sigma^0$ set is $\nu^*(\Sigma^0)=1$ and hence $\nu^*(\Sigma\setminus\Sigma^0)=0$. However, if $\alpha=R<R_{\max}$, then \eqref{all3a} implies that the probability ``mass" on $\Sigma^0$ set is $\nu^*(\Sigma^0)<1$ and hence equations \eqref{all3b}-\eqref{all3c} are employed. While $\nu^*(\Sigma_0)>0$, \eqref{all3c} implies that $\nu^*(\Sigma_k)=\mu(\Sigma_k)$ for all $k=1,\dots,r$. However, if $\nu^*(\Sigma_0)=0$, that is, all the probability ``mass" is removed from $\Sigma_0$, then the solution is obtained by moving further into the partition using \eqref{all3c}. For all $R\in[0,2]$, the resulting solution is described via a water-filling effect.

We are now equipped with the solution of maximizing linear functionals with total variation distance ambiguity for both finite, countable alphabets, and abstract alphabet spaces (Polish spaces), and therefore we are ready to apply these results to the dynamic programming recursion under ambiguity on the conditional distribution.

\section{Minimax Stochastic Control with Total Variation Distance Ambiguity}
\label{partially}
In this section, we first introduce the general definition of finite horizon Discounted-Feedback Control Model (D-FCM) with randomized and deterministic control policies, under total variation distance uncertainty (which includes the D-MCM introduced in Section~\ref{idp}), and then we apply the characterization of the maximizing distribution of Section~\ref{abs} to the dynamic programming recursion. In the last section we discuss the infinite horizon D-MCM.
 
 Define $\mathbb{N}^n \triangleq \{0,1,2,\ldots, n\}, n \in \mathbb{N}$.  The state space and the control space are sequences of Polish spaces $\{ {\cal X}_j: j=0, 1, \ldots, n\}$ and $\{ {\cal U}_j: j=0, 1, \ldots, n-1\}$, respectively. These spaces are associated with their corresponding measurable spaces $({\cal X}_{j}, {\cal B}({\cal X}_{j})), \forall j \in \mathbb{N}^{n}$, $( {\cal U}_{j}, {\cal B}({\cal U}_{j}))$, $\forall j \in \mathbb{N}^{n-1}$. Define the product spaces by ${\cal X}_{0,n}{\triangleq} \times_{i=0}^n{\cal X}_i$, ${\cal U}_{0,n-1}{\triangleq} \times_{i=0}^{n-1}{\cal U}_i$, and introduce their product measurable spaces, $({\cal X}_{0,n}, {\cal B}({\cal X}_{0,n}))$, $({\cal U}_{0,n-1}, {\cal B}({\cal U}_{0,n-1}))$, respectively, for $ n \in \mathbb{N}^n$. The state process is denoted by  $x^n \triangleq \{ x_j : j =0,1, \ldots, n \}$, and the control process is denoted by  $u^{n-1}  \triangleq \{ u_j : j =0,1, \ldots, n-1 \}$. For any measurable spaces $({\cal X}, {\cal B}({\cal X})), ({\cal Y}, {\cal B}({\cal Y}))$, the set of stochastic Kernels on $({\cal Y}, {\cal B}({\cal Y}))$ conditioned on $({\cal X}, {\cal B}({\cal X}))$ is denoted by ${\cal Q}({\cal Y}|{\cal X})$.
 
 Given $({\cal X}_{0,n},{\cal B}({\cal X}_{0,n}))$, $({\cal U}_{0,n-1},{\cal B}({\cal U}_{0,n-1}))$ the Borel state and control or action spaces, respectively, and the initial state distribution $\nu_0(dx_0)$, we introduce the space $H_{0,n}$ of admissible observable histories by
\begin{align*}
  H_{0,n}\triangleq \mathbb{K}_0\times \mathbb{K}_1\times \hdots \times \mathbb{K}_{n-1}\times {\cal X}_n\equiv \times_{i=0}^{n-1}\mathbb{K}_i\times {\cal X}_n,\quad n\in\mathbb{N}, \quad H_{0,0}={\cal X}_0
\end{align*}
 where $\mathbb{K}_i\triangleq\left\{(x_i,u_i):x_i\in{\cal X}_i,u_i\in{\cal U}_i(x_i)\right\}$, denote the feasible state-action pairs, for $i=0,1,\hdots,n-1$. A typical element $h_{0,n}\in H_{0,n}$ is a sequence of the form \begin{align*}
  h_{0,n}=(x_0,u_0,\hdots,x_{n-1},u_{n-1},x_n),\quad (x_i,u_i)\in \mathbb{K}_i,\quad i=0,\hdots,n-1,\quad x_n\in{\cal X}_n.
\end{align*}
Similarly, introduce 
\begin{align*}
  G_{0,n}&={\cal X}_0\times{\cal U}_0\times\hdots\times{\cal X}_{n-1}\times{\cal U}_{n-1}\times{\cal X}_n\equiv\times_{i=0}^{n-1}({\cal X}_i\times{\cal U}_i)\times{\cal X}_n,\quad n\in\mathbb{N}\\
  G_{0,0}&=H_{0,0}={\cal X}_0.
\end{align*}
The spaces $G_{0,n}$ and $H_{0,n}$ are equipped with the natural $\sigma$-algebra ${\cal B}(G_{0,n})$ and ${\cal B}(H_{0,n})$, respectively.

Next, we give the precise definition of discounted feedback control model.\\

\begin{definition}
\label{dfcm}
A finite horizon D-FCM is a septuple
\begin{equation}
\begin{multlined}
\mbox{D-FCM}:\Big({\cal X}_{0,n},  {\cal U}_{0,n-1}, \{ {\cal U}_i(x_i): x_i \in {\cal X}_i\}_{i=0}^{n-1},\{ Q_i(dx_i| x^{i-1}, u^{i-1}):\\
  (x^{i-1}, u^{i-1}) \in {\cal X}_{0,i-1} \times {\cal U}_{0,i-1} \}_{i=0}^{n},\{f_i\}_{i=0}^{n-1}, h_n,\alpha\Big)
\end{multlined}\end{equation}
consisting of the items \textbf{(a)}-\textbf{(c)}, \textbf{(e)}-\textbf{(g)} of finite horizon D-MCM (\ref{mcm}), while the controlled process distribution in \textbf{(d)} is replaced by the non-Markov collection  $\{ Q_i(dx_i| x^{i-1}, u^{i-1}): (x^{i-1}, u^{i-1}) \in \times_{j=0}^{i-1}\mathbb{K}_j \}_{i=0}^{n}$.
\end{definition}

Next, we give the definitions of randomized, deterministic, and stationary control strategies or policies.\\

\begin{definition}
\label{policies}
A randomized control strategy is a sequence $\pi \triangleq \{ \pi_0, \ldots, \pi_{n-1} \}$ of stochastic kernels $\pi_i (\cdot| \cdot)$ on $({\cal U}_i, {\cal B}({\cal U}_i))$ conditioned on $(H_{0,i},{\cal B}(H_{0,i}))$ (e.g., $\pi_i (du_i| x^{i}, u^{i-1})$ ) satisfying
\bes
\pi_i({\cal U}_i(x_i)| x^{i}, u^{i-1})=1\quad \mbox{for every}\quad (x^{i}, u^{i-1}) \in H_{0,i},\quad i=0,1,\hdots,n-1.\ees
 The set of all such policies is denoted by $\mathbf{\Pi}_{0,n-1}$.
 
A strategy $\pi \triangleq \{\pi_i: i=0,\ldots, n-1\} \in \mathbf{\Pi}_{0,n-1}$ is called\\
\textbf{(a)} {\it  deterministic feedback  strategy} if there exists a sequence $g \triangleq \{g_j:j=0,1,\ldots,n-1\}$ of measurable functions  $g_j: \times_{i=0}^{j-1} {\mathbb K}_i \times  {\cal X}_{j} \longmapsto {\cal U}_j$, such that for all  $\left(x^{j}, u^{j-1}\right) \in H_{0,j}$, $j \in \mathbb{N}^{n-1}$,  $g_j(x_0,u_0,x_1,u_1,\ldots,x_{j-1},u_{j-1}, x_j ) \in {\cal U}_j(x_j)$,  and  $\pi_j\left(\cdot| x^{j}, u^{j-1}\right)$ assigns mass 1 to some point in ${\cal U}_j$, that is,
\bes
\pi_i\left(A_i| x^{i}, u^{i-1}\right)=I_{A_i}\left(g_i\left(x^{i}, u^{i-1}\right)\right), \quad \forall A_i \in {\cal B}({\cal U}_i),\quad i=0,1,\hdots,n-1,
\ees
where $I_{A_i}(\cdot)$ is the indicator function of $A_i \in {\cal B}({\cal U}_i)$. \\
The set of deterministic feedback strategies is denoted by $\mathbf{\Pi}_{0,n-1}^{DF}$;\\
\textbf{(b)} {\it  deterministic Markov   strategy} if there exists a sequence $g \triangleq \{g_j:j=0,1,\ldots,n-1\}$ of measurable functions  $g_j:   {\cal X}_{j} \rightarrow {\cal U}_j$ satisfying   $g_j(x_j) \in {\cal U}_j(x_j)$  for all  $x_{j} \in   {\cal X}_{j}$, $j \in \mathbb{N}^{n-1}$,    and  $\pi_j(\cdot| x^{j}, u^{j-1})$ is concentrated at $g_j(x_j) \in {\cal U}_j(x_j)$ for all $(x^{j}, u^{j-1}) \in H_{0,j}$, $j \in \mathbb{N}^{n-1}$.\\
The set of deterministic Markov strategies is denoted by $\mathbf{\Pi}_{0,n-1}^{DM}$;\\
\textbf{(c)} {\it  deterministic stationary Markov strategy} if there exists a measurable function $g:{\cal X}\longrightarrow {\cal U}$ such that $g(x_t)\in {\cal U}(x_t)$, $\forall x_t\in {\cal X}$, and $\pi_j(\cdot|x^j,u^{j-1})$ assigns mass to some point $u_j$, $\forall (x^j,u^{j-1})\in H_{0,j}$, e.g., \begin{align*}
  \pi_i(A_i|x^i,u^{i-1})=I_{A_i}(g(x_i)),\quad \forall A_i\in {\cal B}({\cal U}_i),\quad i=0,\hdots,n-1.
\end{align*}
The set of deterministic stationary Markov strategies is denoted by $\mathbf{\Pi}_{0,n-1}^{DS}$.
\end{definition}

According to Definition \ref{policies}, the set of control policies is non-empty, since we have assumed existence of measurable functions $g_j:{\mathbb K}_{0,j-1}\times{\cal X}_j\longrightarrow{\cal U}_j$ such that $\forall x^j,u^{j-1}\in {\mathbb K}_{0,j-1}\times {\cal X}_j$, $g_j(x^j,u^{j-1})\in {\cal U}_j({\cal X}_j), \forall j\in \mathbb{N}^{n-1}$. Sufficient conditions for this to hold are in general obtained via measurable selection theorems \cite{hern1996discrete}. For denumerable set (countable alphabet) ${\cal X}_j$ endowed with the discrete topology any function is measurable.
 Given a controlled process $ \left\{Q_i(\cdot|x^{i-1},u^{i-1}):(x^{i-1},u^{i-1})\in {\mathbb K}_{0,i-1}\right\}_{i=0}^n$
 and a randomized control process $ \left\{\pi_i(\cdot|x^{i},u^{i-1}):(x^{i},u^{i-1})\in {\mathbb K}_{0,i-1}\times {\cal X}_i\right\}_{i=0}^n\in\mathbf{\Pi}_{0,n-1}$
 and the initial probability $\nu_0(\cdot)\in {\cal M}_1({\cal X}_0)$, then by Ionescu-Tulceu theorem \cite{Bertsekas:2007:SOC:1512940} there exists a unique probability measure $\mathbf{Q}^{\pi}_{\nu}$ on $(\Omega,{\cal F})$ defined by
 \begin{equation}
 \begin{multlined} \label{upm1}\mathbf{Q}^{\pi}_{\nu}(dx_0,du_0,dx_1,du_1,\hdots,dx_{n-1},du_{n-1},dx_n)=Q_0(dx_0)\pi_0(du_0|x_0)\\ \otimes Q_1(dx_1|x_0,u_0)
 \pi_1(du_1|x^1,u_0)\otimes\hdots\otimes Q_{n-1}(dx_{n-1}|x^{n-2},u^{n-2})\\ \pi_{n-1}(du_{n-1}|x^{n-1},u^{n-2})\otimes Q_n(dx_n|x^{n-1},u^{n-1})
\end{multlined}  \end{equation}
such that
\begin{eqnarray*}
&&\mathbf{Q}^{\pi}_{\nu}(x_0\in A)=\nu(A),\quad\mbox{$ A\in{\cal B}({\cal X}_0)$}\\
&&\mathbf{Q}^{\pi}_{\nu}(u_j\in B|h_{0,j})=\pi_t(B|h_{0,j}),\quad\mbox{$ B\in{\cal B}({\cal U}_j)$}\\
&&\mathbf{Q}^{\pi}_{\nu}(x_{j+1}\in C|h_{0,j},u_t)=Q(C|h_{0,j},u_j),\quad\mbox{$ C\in {\cal B}({\cal X}_{j+1})$}.
\end{eqnarray*}
 Given the sample pay-off
 \begin{align} \label{samplepoff} F^\alpha_{0,n}(x_0,u_0,x_1,u_1,\hdots,x_{n-1},u_{n-1},x_n)\triangleq\sum_{j=0}^{n-1}\alpha^jf_j(x_j,u_j)+\alpha^nh_n(x_n)
 \end{align}
 its expectation is
 \begin{equation}
\begin{multlined}
{\mathbb E}_{\mathbf{Q}_{\nu}^{\pi}}\left\{ F^\alpha_{0,n}(x_0,u_0,x_1,u_1,\hdots,x_{n-1},u_{n-1},x_n)\right\}=
\int F^\alpha_{0,n}(x_0,u_0,x_1,u_1,\hdots,\\
x_{n-1},u_{n-1},x_n)
\mathbf{Q}^{\pi}_{\nu}(dx_0,du_0,dx_1,du_1,\hdots,dx_{n-1},du_{n-1},dx_n).
\end{multlined}\end{equation}
Note that the class of randomized strategies $\mathbf{\Pi}_{0,n-1}$ embeds deterministic feedback and Markov strategies.

\subsection{Variation Distance Ambiguity}
\label{subformulation}
Next, we introduce the definitions of nominal controlled process distributions (for finite horizon D-FCM and D-MCM), and their corresponding ambiguous controlled process distributions.

For each $\pi\in \mathbf{\Pi}_{0,n-1}^{DF} $, $\pi\in \mathbf{\Pi}_{0,n-1}^{DM} $ and $\pi\in \mathbf{\Pi}_{0,n-1}^{DS} $ the nominal controlled process is described by a sequence of conditional distributions as follows.\\

\begin{definition} (Nominal Controlled Process Distributions).
\label{dt}
A nominal controlled state processes $\{x^g=x_0^g, x_1^g, \ldots, x_n^g: \pi \in \mathbf{\Pi}_{0,n-1}^{DF},    \hso \pi \in \mathbf{\Pi}_{0,n-1}^{DM},\hso \mbox{or}\hso \pi \in \mathbf{\Pi}_{0,n-1}^{DS}\}$  corresponds to  a sequence of stochastic kernels as follows:\\
\textbf{(a)} {\it Feedback Controlled Process.} 
\bes \mbox{For every}\hso A \in {\cal B}({\cal X}_{j}),\hso
Prob(x_{j} \in A| x^{j-1},u^{j-1}   )= Q^o_{j}(A| x^{j-1}, u^{j-1})
\ees
where $Q^o_{j}(A| x^{j-1}, u^{j-1}) \in {\cal Q}({\cal X}_{j} | \mathbb{K}_{0,j-1}), \forall j\in \mathbb{N}_+^{n}$.\\
\textbf{(b)}  {\it  Markov Controlled Process.}
\begin{align*}\mbox{For every}\hso A \in {\cal B}({\cal X}_{j}),\hso
Prob(x_j \in A| x^{j-1}, u^{j-1}    )= Q^o_j(A| x_{j-1}, u_{j-1})
\end{align*}
where $Q^o_{j}(A| x_{j-1}, u_{j-1}) \in {\cal Q}({\cal X}_{j} | \mathbb{K}_{j-1}), \forall j\in \mathbb{N}_+^{n}$.\\
\textbf{(c)}  {\it  Stationary Markov Controlled Process.}
\begin{align*}\mbox{For every}\hso A \in {\cal B}({\cal X}),\hso
Prob(x_j \in A| x^{j-1}, u^{j-1}    )= Q^o(A| x_{j-1}, u_{j-1})
\end{align*}
where $Q^o(A|x_{j-1}, u_{j-1}) \in {\cal Q}({\cal X} | \mathbb{K})$.\\
\end{definition}

The class of controlled processes is described by the sequence of stochastic kernels, \bes\{{ Q}_{j}(dx_{j}|x^{j-1},u^{j-1})\in  {\cal Q}({\cal X}_{j} | \mathbb{K}_{0,j-1}: j=0, \ldots, n\}\ees belonging to a total variation distance set as follows.\\

\begin{definition}(Class of Controlled Process Distribution)
\label{type2}
Given a nominal controlled process stochastic kernel of Definition \ref{dt}, and  $R_i \in [0,2], 0\leq i\leq n$ the class of controlled process stochastic kernels is defined as follows: \\
\textbf{ (a)} {\it Class with respect to Feedback Nominal Controlled Process.} \\
 Given a fixed $Q^o_{{j}}(\cdot| x^{j-1}, u^{j-1}) \in {\cal Q}({\cal X}_{j} | \mathbb{K}_{0,j-1})$, $j=0,1,\hdots,n$ the class of stochastic kernels is defined by
\begin{multline*}
  \mathbf{B}_{R_i}(Q^o_{{i}}) (x^{i-1},u^{i-1}) \triangleq \Big\{ { Q}_{{i}}(\cdot| x^{i-1},u^{i-1}) \in {\cal Q}({\cal X}_i|\mathbb{K}_{0,i-1}):\\
  || { Q}_{{i}}(\cdot| x^{i-1},u^{i-1}) -Q^o_{{i}}(\cdot| x^{i-1},u^{i-1}) ||_{TV}  \leq R_i \Big\},\quad i=0,1,\ldots, n. 
\end{multline*}
\textbf{ (b)} {\it Class with respect to Markov Nominal Controlled Process.}\\
 Given a fixed $Q^o_{{j}}(\cdot| x_{j-1}, u_{j-1}) \in {\cal Q}({\cal X}_{j} |  \mathbb{K}_{j-1})$, $j=0,1,\hdots,n$ the class of stochastic kernels is defined by
\begin{multline*}
  \mathbf{B}_{R_i}(Q^o_{{i}})  (x^{i-1},u^{i-1}) \triangleq \Big\{ { Q}_{{i}}(\cdot| x^{i-1},u^{i-1}) \in {\cal Q}({\cal X}_i|\mathbb{K}_{0,i-1}): \\
     || { Q}_{{i}}(\cdot| x^{i-1},u^{i-1}) -Q^o_{{i}}(\cdot| x_{i-1},u_{i-1}) ||_{TV}  \leq R_i \Big\},\quad i=0,1,\ldots, n. 
\end{multline*}
\textbf{ (c)} {\it Class with respect to Stationary Markov Nominal Controlled Process.}\\
 Given a fixed $Q^o(\cdot| x_{j-1}, u_{j-1}) \in {\cal Q}({\cal X} , \mathbb{K})$ the class of stochastic kernels is defined by
\bes
  \mathbf{B}_{R}(Q^0)  (x,u) \triangleq \Big\{ Q(\cdot| x,u) \in {\cal Q}({\cal X}|\mathbb{K}): 
     || Q(\cdot| x,u) -Q^o(\cdot| x,u) ||_{TV}  \leq R \Big\}.
\ees
\end{definition}
Note that in Definition \ref{type2} (a), (b), although we use the same notation $\mathbf{B}_{R_i}(Q^o_{{i}})  (x^{i-1},u^{i-1})$ these sets are different because the nominal distribution $Q^o_i(\cdot|\cdot)$ can be of Feedback or Markov form. The above model is motivated by the fact that dynamic programming involves conditional expectation with respect to the collection of conditional distributions $\{{ Q}_{i}(\cdot| x^{i-1},u^{i-1})$ $\in {\cal Q}({\cal X}_i|\mathbb{K}_{0,i-1}): i=0, \ldots, n\}$. Therefore, any ambiguity in these distributions will affect the optimality of the strategies.

\subsection{Pay-Off Functional}

  For each $\pi \in \mathbf{\Pi}_{0,n-1}^{DF}$ or $\pi \in \mathbf{\Pi}_{0,n-1}^{DM}$ the average pay-off is defined by
\bea
J_{0,n}(\pi,Q_i:i=0,\hdots,n)\triangleq {\mathbb E}_{\mathbf{Q}^{\pi}_{\nu}}\Big\{ \sum_{j=0}^{n-1} \alpha^jf_j(x_j,u_j) + \alpha^nh_n(x_n)\Big\} \label{dt2}
\eea
where ${\mathbb E}_{\mathbf{Q}^{\pi}_{\nu}} \{\cdot\}$ denotes expectation with respect to the true joint measure $\mathbf{Q}_{\nu}^{\pi}(dx^n,du^{n-1})$ defined by (\ref{upm1}) such that $Q_i(\cdot|x^{i-1},u^{i-1})\in \mathbf{B}_{R_i}(Q^o_i)$, $i=0,1,\hdots,n$ (e.g., it belongs to the total variation distance ball of Definition \ref{type2}).

Next, we introduce assumptions so that the maximization over the class of ambiguous measures is well-defined.\\

\begin{assumption}
\label{ass4.1}
The nominal system family satisfies the  following assumption: The maps $\{f_{j}: {\cal X}_j \times {\cal U}_j \longmapsto \mathbb{R}: j=0, 1, \ldots, n-1\}$, $h_{n}: {\cal X}_n \longmapsto \mathbb{R}$  are bounded, continuous and non-negative.
\end{assumption}

Note that it is possible to relax Assumption \ref{ass4.1} to lower semi-continuous non-negative functions bounded from below.

\subsection{Minimax Dynamic Programming for Finite Horizon D-FCM and D-MCM}
\label{dp}
In this section we shall apply the results of Section \ref{abs} to formulate and solve minimax stochastic control under a) finite horizon D-FCM ambiguity, and b) finite horizon D-MCM ambiguity.

\subsubsection{Dynamic Programming for Finite Horizon D-FCM Subject to Ambiguity}
Utilizing the above formulation, next we define the minimax stochastic control problem, where the maximization is over a total variation distance ball, centered at the nominal conditional distribution $Q^o_{i}(dx_i| x^{i-1},u^{i-1})  \in {\cal Q}({\cal X}_{i}| \mathbb{K}_{0,i-1})$ having radius $R_i \in [0,2]$, for $i=0,1,\ldots, n$.\\

\begin{problem}
\label{def4.2nn}
Given a nominal feedback controlled process of Definition~\ref{dt} (a), an admissible policy set $\mathbf{\Pi}_{0,n-1}^{DF}$ and an ambiguity class  $\mathbf{B}_{R_{k}}(Q^o_{k})(x^{k-1},u^{k-1}), k{=}0,...,n$ of Definition \ref{type2} (a), find a ${\pi }^* {\in} \mathbf{\Pi}_{0,n-1}^{DF}$ and a sequence of stochastic kernels
${ Q}_{k}^{*}(dx_{k}| x^{k-1}, u^{k-1})  \in \mathbf{B}_{R_{k}}(Q^o_{k})(x^{k-1},u^{k-1})$, $k=0,1,...,n$  which solve the following minimax optimization problem.
\begin{equation}
\begin{multlined}
  J_{0,n}(\pi^*,{Q}^{*}_{k}:k=0,\hdots,n)  
  = \inf_{\pi \in \mathbf{\Pi}_{0,n-1}^{DF}}\Big\{ \\
   \sup_{{\sr  {{ Q}_{k}(\cdot| x^{k-1}, u^{k-1})  \in \mathbf{B}_{R_k}(Q^o_k)(x^{k-1},u^{k-1})}{ k=0,1,\ldots, n}}   } {\mathbb E}_{\mathbf{Q}^{\pi}_{\nu}} \Big\{\sum_{k=0}^{n-1} \alpha^kf_k(x_k^g,u_k^g) + \alpha^nh_n(x_n^g) \Big\}\Big\}.\label{eq4.20a}
\end{multlined}
\end{equation}
\end{problem}
Next, we apply dynamic programming to characterize the solution of (\ref{eq4.20a}), by first addressing the maximization. Define the pay-off associated with the maximization problem
\begin{equation*}
J_{0,n}(\pi, {Q}^{*}_{k}:k=0,\hdots,n)\triangleq   \sup_{       {\sr  {{ Q}_{k}(\cdot| x^{k-1},u^{k-1})  \in \mathbf{B}_{R_k}(Q^o_k)(x^{k-1},u^{k-1})}{ k=0,1,\ldots, n}}   } J_{0,n}(\pi,Q_{k}:k=0,\hdots,n).
\end{equation*}
For a given $\pi\in\mathbf{\Pi}_{0,n-1}^{DF}$, which defines $\{g_j:j=0,\ldots,n-1 \}$, and $\pi_{[k,m]}\equiv u^g_{[k,m]}$, denoting the restriction of policies in $[k,m]$, $0\leq k\leq m\leq n-1$, define the conditional expectation taken over the events ${\cal G}_{0,j}\triangleq \sigma\{x_0^g, \ldots, x_j^g, u_0^g, \ldots, u_j^g\}$ maximized over the class $\mathbf{B}_{R_k}(Q^o_k)( x^{k-1}, u^{k-1})$, $k=j+1,\ldots, n$, as follows \cite{caines88,varayia86}:
\begin{equation}
\begin{multlined}\label{eq4.2300aa}
V_j(u^g_{[j,n-1]},{\cal G}_{0,j})\triangleq \sup_{       {\sr  {Q_{k}(\cdot| x^{k-1}, u^{k-1})  \in \mathbf{B}_{R_k}(Q^o_k)( x^{k-1}, u^{k-1})}{ k= j+1,\ldots, n}}}  \mathbb{E}_{\mathbf{Q}^{\pi}_{\nu}} \Big\{\sum_{k=j}^{n-1} \alpha^kf_k(x_k^g,u_k^g) \\+ \alpha^nh_n(x_n^g) | {\cal G}_{0,j}  \Big\}
\end{multlined}
\end{equation}
where $\mathbb{E}_{\mathbf{Q}^{\pi}_{\nu}}\{\cdot|{\cal G}_{0,j}\}$ denotes conditional expectation with respect to ${\cal G}_{0,j}$ calculated on the probability measure $\mathbf{Q}^{\pi}_{\nu}$. Then, $V_j(u^g_{[j,n-1]},{\cal G}_{0,j})$ satisfies the following dynamic programming equation \cite{varayia86},
\begin{eqnarray}
V_n({\cal G}_{0,n})&=&\alpha^nh_n(x_n^g)\label{eq4.230a}\\
V_j(u^g_{[j,n-1]},{\cal G}_{0,j})&=& \sup_{  {Q_{j+1}(\cdot| x^{j}, u^{j})  \in \mathbf{B}_{R_{j+1}}(Q^0_{j+1})(x^j,u^j)}}\Big\{\nonumber\\[-1.5ex]\label{eq4.230}\\[-1.5ex]
&&\mathbb{E}_{{Q}_{j+1}(\cdot| x^{j}, u^{j})}\Big\{\alpha^jf_j(x_j^g,u_j^g)+V_{j+1}(u^g_{[j+1,n-1]},{\cal G}_{0,j+1})  \Big\}\Big\}\nonumber
\end{eqnarray}
 where $\mathbb{E}_{{Q}_{j+1}(\cdot| x^j, u^j)}\{\cdot\}$ denotes expectation with respect to ${Q}_{j+1}(dx_{j+1}| \mathbb{K}_{0,j})$.

Next, we present the dynamic programming recursion for the minimax problem. Let $V_j({\cal G}_{0,j})$ represent the minimax pay-off on the future time horizon $\{j,j+1,...,n\}$  at time $j\in {\mathbb N}^n_+$ defined by
\begin{equation}
\begin{multlined}
\label{eq4.201}
 V_j({\cal G}_{0,j}) \triangleq \inf_{\pi \in \mathbf{\Pi}_{j,n-1}^{DF}} \sup_{       {\sr  {Q_{k}(\cdot| x^{k-1}, u^{k-1})  \in \mathbf{B}_{R_k}(Q^0_{k})( x^{k-1}, u^{k-1})}{ k=j+1,\ldots, n}}   } \Big\{    \\ \mathbb{E}_{\mathbf{Q}^{\pi}_{\nu}} \Big\{\sum_{k=j}^{n-1} \alpha^kf_k(x_k^g,u^g_k) + \alpha^nh_n(x_n^g) | {\cal G}_{0,j}  \Big\} \Big\}
 = \inf_{\pi \in \mathbf{\Pi}_{j,n-1}^{DF}}V_j(u^g_{[j,n-1]},{\cal G}_{0,j}).
 \end{multlined}
\end{equation}
Then by reconditioning we obtain
\begin{equation}
\begin{multlined}
  V_j({\cal G}_{0,j}) \triangleq \inf_{u \in {\cal U}_{ad}[j,n-1]} \sup_{       {\sr  {Q_{k}(\cdot| x^{k-1}, u^{k-1})  \in \mathbf{B}_{R_k}(Q^0_{k})( x^{k-1}, u^{k-1} )}{ k=j+1,\ldots, n}}   } \Big\{ 
   \\ \mathbb{E}_{\mathbf{Q}^{\pi}_{\nu}} \Big\{ \alpha^jf_j(x_j^g,u^g_j)  
    + \mathbb{E}_{\mathbf{Q}^{\pi}_{\nu}} \Big\{  \sum_{k=j+1}^{n-1} \alpha^kf_k(x_k^g,u^g_k)  +\alpha^n h_n(x_n^g) | {\cal G}_{0,j+1}   \Big\}  | {\cal G}_{0,j} \Big\}  \Big\}.
\label{eq4.202}
\end{multlined}
\end{equation}
Hence, we deduce the following dynamic programming recursion
\begin{eqnarray}
V_n({\cal G}_{0,n})&=&\alpha^nh_n(x_n^g)\label{eq4.205}\\
\label{eq4.204} V_j({\cal G}_{0,j}) &\triangleq & \adjustlimits \inf_{u_j \in {\cal U}_{j}(x)} \sup_{ Q_{j+1}(\cdot|x^{j}, u^{j})  \in \mathbf{B}_{R_{j+1}}(Q^0_{j+1})(  x^{j}, u^{j})  } \Big\{   \nonumber\\ [-1.5ex]\\[-1.5ex] 
 && \mathbb{E}_{{Q}_{j+1}(\cdot| x^j, u^j)} \Big\{ \alpha^jf_j(x_j^g,u_j^g)  +  V_{j+1}({\cal G}_{0,j+1})   \Big\}\Big\}.\nonumber  
\end{eqnarray}
By applying the results of Section \ref{abs} to (\ref{eq4.205}), (\ref{eq4.204}) we obtain the following theorem.\\

\begin{theorem}\label{th411}
Suppose there exist an optimal policy for Problem \ref{def4.2nn}, and assume $V_{j+1}(\cdot){:} {\cal X}_{0,j+1} {\times} {\cal U}_{0,j}{\longmapsto} [0,\infty)$ in (\ref{eq4.201}) is bounded continuous in $x {\in} {\cal X}_{j+1}$, $j=0,\hdots,n{-}1$.

 1) The dynamic programming recursion is given by
\begin{eqnarray}
V_n({\cal G}_{0,n})&=&\alpha^nh_n(x_n^g)\label{eq4.230aaa}\\
V_j({\cal G}_{0,j})&=&\inf_{u_j\in {\cal U}_j(x)}\Big\{\mathbb{E}_{Q^o_{j+1}}\Big(\alpha^jf_j(x_j^g,u_j^g)
 +V_{j+1}({\cal G}_{0,j+1})|{\cal G}_{0,j}\Big)\nonumber\\[-1.5ex]\label{eq4.230aa}\\[-1.5ex]
 &&+\frac{R_j}{2}\Big(\sup_{x_{j+1} \in {\cal X}_{j+1}} V_{j+1}({\cal G}_{0,j+1})-\inf_{x_{j+1} \in {\cal X}_{j+1}} V_{j+1}({\cal G}_{0,j+1})\Big)\Big\}. \nonumber
\end{eqnarray}
Moreover,
\begin{equation}
V_j({\cal G}_{0,j})=\inf_{u_j\in {\cal U}_j(x)}\mathbb{E}_{Q^*_{j+1}}\Big\{\alpha^jf_j(x_j^g,u_j^g)
 +V_{j+1}({\cal G}_{0,j+1})|{\cal G}_{0,j}\Big\}
\end{equation}
where, the optimal conditional distributions $\{Q_j^*:j=0,1,\hdots,n-1\}$ are given by
\begin{eqnarray}
&&Q^*_{j+1}\left({\cal X}^+_{j+1}|x^j,u^j\right)=Q^o_{j+1}({\cal X}_{j+1}^+| x^{j}, u^j ) + \frac{R_{j+1}}{2}\in[0,1],\ \ \mbox{$(x^j,u^j){\in} \mathbb{K}_{0,j}$} \label{pc1} \\
&&Q^*_{j+1}\left({\cal X}^-_{j+1}|x^j,u^j\right)=Q^o_{j+1}({\cal X}^-_{j+1}| x^j, u^j) - \frac{R_{j+1}}{2}\in[0,1],\ \ \mbox{$(x^j,u^j){\in} \mathbb{K}_{0,j}$} \label{pc1a}\\
&&Q^*_{j+1}\left(A|x^j,u^j\right)=Q^o_{j+1}(A| x^j, u^j),\ \forall A{\subseteq}  {\cal X}_{j+1}{\setminus} {\cal X}^+_{j+1}{\cup}{\cal X}^-_{j+1},\ \ \mbox{$(x^j,u^j){\in} \mathbb{K}_{0,j}$}\label{pc1aa}
\end{eqnarray}
and \footnote{Note the notation $\Sigma^0$ and $\Sigma_0$ in Section \ref{abs} is identical to the notation ${\cal X}^+_{j+1}$ and ${\cal X}^-_{j+1}$, respectively.}
\begin{eqnarray}
 {\cal X}_{j+1}^+  &\triangleq & \Big\{ x_{j+1} \in {\cal X}_{j+1}{:}V_{j+1}({\cal G}_{0,j+1}) {=} \sup \Big\{ V_{j+1}({\cal G}_{0,j+1}){:} x_{j+1} {\in} {\cal X}_{j+1} \Big\}  \Big\} \\
   {\cal X}_{j+1}^- &\triangleq & \Big\{ x_{j+1} \in {\cal X}_{j+1}{:}   V_{j+1}({\cal G}_{0,j+1}){=} \inf \Big\{ V_{j+1}({\cal G}_{0,j+1}){:} x_{j+1} {\in} {\cal X}_{j+1} \Big\}  \Big\}. \label{pc2}
\end{eqnarray}

2) The total pay-off is given by
\begin{equation} J_{0,n}(\pi^*,{ Q}^{\ast}_{i}:i=0,\hdots,n-1)= \sup_{ Q_0(\cdot)  \in \mathbf{B}_{R_0}(Q^o)}  \mathbb{E}_{Q_{0}}   \Big\{V_0({\cal G}_{0,0})\Big\}. \label{eq4.230aaaaa}
\end{equation}
\end{theorem}
\begin{proof}
 1) Consider (\ref{eq4.204}) expressed in integral form
\begin{equation} \begin{multlined}
V_j({\cal G}_{0,j})=\inf_{u_j\in {\cal U}_{j}(x)}\Big\{\alpha^jf_j(x_j,u_j)\\+\sup_{Q_{j+1}(\cdot|x^j,u^j)\in \mathbf{B}_{R_{j+1}}(Q^o_{j+1})(  x^{j}, u^{j})}
\int V_{j+1}({\cal G}_{0,j},z)Q_{j+1}(dz|x^j,u^j)\Big\}.
\end{multlined}\end{equation}
 By applying (\ref{all3}) we obtain (\ref{eq4.230aaa}), (\ref{eq4.230aa}), while (\ref{pc1})-(\ref{pc2}) follow as well.
 
 2) By evaluating (\ref{eq4.201}) at $j=0$ we obtain (\ref{eq4.230aaaaa}). This completes the derivation.\hfill 
\end{proof}
\\

By Theorem \ref{th411}, the maximizing measure is given by (\ref{pc1})-(\ref{pc1aa}), and it is a functional of the nominal measure. At this stage we cannot claim that the maximizing measure is Markovian, and hence the optimal strategy is not necessarily Markov. Therefore, the computation of optimal strategies using non-Markov nominal controlled processes is computationally intensive. Next, we restrict the minimax formulation to Markov controlled nominal processes. 

\subsubsection{Dynamic Programming for Finite Horizon D-MCM Subject to Ambiguity} 
Consider the Markov nominal controlled processes, based on Definition~\ref{dt} (b), and define
\begin{equation*}
V_j(u^g_{[j,n-1]},{\cal G}_{0,j})\triangleq \hspace{1.5cm}\smashoperator{\sup_{       {\sr  {Q_{k}(\cdot| x^{k-1}, u^{k-1})  \in \mathbf{B}_{R_k}(Q^o_k)( x_{k-1}, u_{k-1})}{ k= j+1,\ldots, n}}}}  \hspace{1.6cm}\mathbb{E}_{\mathbf{Q}^{\pi}_{\nu}} \Big\{\sum_{k=j}^{n-1} \alpha^kf_k(x_k^g,u_k^g) + \alpha^nh_n(x_n^g) | {\cal G}_{0,j}  \Big\}.
\end{equation*}
In view of Section~\ref{abs}, specifically, the relation between the maximizing distribution and the nominal distribution (\ref{f2n})-(\ref{cond1}), which also apply to conditional distributions, we deduce that the maximization conditional distribution $Q_i^*(dx_i| x^{i-1}, u^{i-1})$ is Markovian, hence $Q_i^*(dx_i| x^{i-1}, u^{i-1})=Q_i^*(dx_i| x_{i-1}, u_{i-1})$, $\forall (x^{i-1},u^{i-1})\in \mathbb{K}_{0,i-1}$. This observation can be verified by checking expressions (\ref{pc1})-(\ref{pc1aa}). Then we define 
\begin{equation}\label{eq4.2300aamcm1}
V_j(u^g,x)\triangleq \hspace{1.6cm}\smashoperator{\sup_{       {\sr  {Q_{k}(\cdot| x_{k-1}, u_{k-1})  \in \mathbf{B}_{R_k}(Q^o_k)( x_{k-1}, u_{k-1})}{ k= j+1,\ldots, n}}} }\hspace{1.6cm} \mathbb{E}_{\mathbf{Q}^{\pi}_{\nu}} \Big\{\sum_{k=j}^{n-1} \alpha^kf_k(x_k^g,u_k^g) {+} \alpha^nh_n(x_n^g) | x  \Big\}.
\end{equation}
Utilizing the above observations we obtain the analog of Theorem \ref{th411} for finite horizon D-MCM, as follows.

Define the value function
\begin{equation}
 \label{vf1a}V_j(x){=}\inf_{\pi \in \mathbf{\Pi}_{j,n-1}^{DM}}\hspace{2.2cm}\smashoperator{\sup_{{\sr  {{ Q}_{k}(\cdot| x_{k-1}, u_{k-1})  \in \mathbf{B}_{R_k}(Q^o_k)(x_{k-1},u_{k-1})}{ k=j+1,\ldots, n}}   }} \hspace{1.7cm}
  {\mathbb E}_{\mathbf{Q}_\nu^\pi} \Big\{\sum_{k=j}^{n-1} \alpha^kf_k(x_k^g,u_k^g){ +} \alpha^nh_n(x_n^g)|x\Big\}. \end{equation}
Then we obtain the following theorem.\\

\begin{theorem}
\label{th4.111}
 Suppose there exists an optimal policy for Problem \ref{def4.2nn}, for the class of Markov Nominal Controlled Process of Definition~\ref{type2} (b). Then the following hold:
 
1) If the infimum over feedback strategies in (\ref{vf1a}) exists it is Markov $\pi\in \mathbf{\Pi}_{0,n-1}^{DM}$.

2) The value function $V_j(x)$ satisfies the dynamic programming recursion
\begin{eqnarray}
V_n(x)&=&\alpha^nh_n(x),\quad\mbox{$ x\in{\cal X}_n$} \label{eq4.2055}\\
V_j(x) &=& \inf_{u \in {\cal U}_j(x)} \sup_{ Q_{j+1}(\cdot| x,u)  \in \mathbf{B}_{R_{j+1}}(Q^o_{j+1})(x,u)  }  \nonumber\\[-1.5ex]\label{eq4.2044}\\[-1.5ex]
&&  \mathbb{E}_{Q_{j+1}(\cdot| x,u)}    \Big\{\alpha^jf_j(x,u)  +  V_{j+1}(x_{j+1})    \Big\},\quad\mbox{$ x\in{\cal X}_j$}.\nonumber
\end{eqnarray}

3) Assume $V_{j+1}(\cdot): {\cal X}_{j+1} \rightarrow [0,\infty)$  is bounded continuous in $x \in {\cal X}_{j+1}$, $j=0,\hdots,n-1$, then the dynamic programming recursion is given by
\begin{eqnarray}
  V_n(x) &=&\alpha^nh_n(x),\quad\mbox{ $x \in {\cal X}_n$}  \label{eq4.2058}\\
 V_j( x)&=& \inf_{u \in {\cal U}_j(x)}\Big\{\alpha^jf_j(x,u)+ \int_{{\cal X}_{j+1}}  V_{j+1}(z) Q^o_{j+1}(dz| x,u) \nonumber\\[-1.5ex]\label{eq4.2057}\\[-1.5ex]
 && + \frac{R_j}{2} \Big( \sup_{z \in{\cal X}_{j+1}} V_{j+1}(z)- \inf_{z \in{\cal X}_{j+1}} V_{j+1}(z)  \Big)  \Big\},\quad\mbox{$ x \in {\cal X}_j $}.\nonumber
\end{eqnarray}
Moreover, 
\begin{equation}
V_j(x)=\inf_{u\in {\cal U}_j(x)}\mathbb{E}_{Q^*_{j+1}}\Big\{\alpha^jf_j(x_j^g,u_j^g)
 +V_{j+1}(x_{j+1})|x_j=x\Big\}
\end{equation}
 where the optimal conditional distribution $\{Q^*_j(\cdot|\cdot,\cdot):j=0,1,\hdots,n-1\}$ is given by
\begin{eqnarray}
&&Q^*_{j+1}\left({\cal X}_{j+1}^+|x_j,u_j\right)=Q^o_{j+1}({\cal X}_{j+1}^+| x_j,u_j) + \frac{R_{j+1}}{2} \in[0,1],\ \ \mbox{$(x_j,u_j)\in \mathbb{K}_j$}\label{pp1a}\\
&&Q^*_{j+1}\left({\cal X}_{j+1}^-|x_j,u_j\right)=Q^o_{j+1}({\cal X}_{j+1}^-| x_j, u_j) - \frac{R_{j+1}}{2}\in[0,1],\ \ \mbox{$(x_j,u_j)\in \mathbb{K}_j$}\label{ppc1}\\
&&Q^*_{j+1}\left(A|x_j,u_j\right)=Q^o_{j+1}(A| x_j, u_j),\ \forall A{\subseteq}  {\cal X}_{j+1}{\setminus} {\cal X}^+_{j+1}{\cup}{\cal X}_{j+1}^-,\ \ \mbox{$(x_j,u_j)\in \mathbb{K}_j$}\label{ppc2}
\end{eqnarray}
and
\begin{eqnarray}
{\cal X}_{j+1}^+  &\triangleq & \Big\{ x_{j+1} \in {\cal X}_{j+1}{:}     V_{j+1}(x_{j+1}){=} \sup \{ V_{j+1}(x_{j+1}){:} x_{j+1} {\in} {\cal X}_{j+1} \}  \Big\}\label{pp11aa}\\
  {\cal X}_{j+1}^-  &\triangleq & \Big\{ x_{j+1} \in {\cal X}_{j+1}{:}   V_{j+1}(x_{j+1}){=} \inf \{ V_{j+1}(x_{j+1}){:} x_{j+1} {\in} {\cal X}_{j+1} \}  \Big\}.\label{ppcc22}
\end{eqnarray}

 4) The total minimax pay-off is
\begin{equation}
J_{0,n}(g^{\ast},\{{Q}^{\ast}_{i}\}^{n}_{i=0}) = \sup_{ Q_0(\cdot)  \in \mathbf{B}_{R_0}(Q^o_0)}
 \mathbb{E}_{Q_{0}}   \Big\{V_0(x_0)\Big\}. \label{eq4.2056}
\end{equation}
\end{theorem}

\begin{proof} 1) Since the nominal controlled process is Markov, from Theorem \ref{th411}, (\ref{pc1})-(\ref{pc1aa}) we deduce that the maximizing measure is also Markov. By the same arguments as in \cite{varayia86} we can show that if the infimum over $u\in \mathbf{\Pi}_{0,n-1}^{DF}$ in (\ref{vf1a}) exists, then it is Markov, and hence $u\in \mathbf{\Pi}_{0,n-1}^{DM}$.

2) By reconditioning we deduce that, the value function satisfies the dynamic programming equation (\ref{eq4.2055}), (\ref{eq4.2044}).

3) By definition, (\ref{eq4.2044}) is also equivalent to 
\begin{equation*}
 V_j(x)=\inf_{u\in {\cal U}(x)}\Big\{\alpha^j f_j(x,u)+\sup_{ Q_{j+1}(\cdot| x,u)  \in \mathbf{B}_{R_{j+1}}(Q^o_{j+1})(x,u)  }\int_{{\cal X}_{j+1}}V_{j+1}(z)Q_{j+1}(dz|x,u)\Big\}.\end{equation*}
Hence, by applying the results of Section \ref{abs} we obtain (\ref{eq4.2058})-(\ref{ppc2}).

4) By evaluating (\ref{vf1a}) at $j=0$ we obtain (\ref{eq4.2056}). This completes the derivation.\hfill
\end{proof}\\


\begin{remark}
\label{remarkg}
 We make the following observations regarding Theorem \ref{th4.111}.\\
(a) The dynamic programming equation (\ref{eq4.2058}), (\ref{eq4.2057}) involves in its right hand side the oscillator seminorm of $V_{j+1}(\cdot)$. \\
(b) The dynamic programming recursion (\ref{eq4.2058}), (\ref{eq4.2057}) can be applied to controlled process with continuous alphabets and to controlled process with finite or countable alphabets, such as Markov Decision models.\\
\end{remark}

Next, we show that for any $j\in \mathbb{N}^{n-1}$, the minimax pay-off $V_j(x)\equiv V_j^{R}(x)$ as a function of $R_j$ is non-decreasing and concave.\\
\begin{lemma}\label{lemmadecrconc}
   Suppose the conditions of Theorem \ref{th4.111} hold and in addition $R_j=R$, $j=1,\hdots,n$. The minimax pay-off $V_j^{R}(x)\equiv V_j(x)$ defined by (\ref{vf1a}) is a non-decreasing concave function of $R$.
\end{lemma}

\begin{proof}
 Consider two values for $R^1, R^2\in \mathbb{R}^+$ such that $0\leq R^1\leq R^2$. Since 
 \begin{equation*}\mathbf{B}_{R^1}(Q^o_k)(x_{k-1},u_{k-1})\subseteq \mathbf{B}_{R^2}(Q^o_k)(x_{k-1},u_{k-1})\end{equation*}
then for every $Q_k(\cdot,x_{k-1},u_{k-1})\in \mathbf{B}_{R^1}(Q^o_k)(x_{k-1},u_{k-1})$ we have $Q_k(\cdot,x_{k-1},u_{k-1})\in B_{R^2}(Q^o_k)(x_{k-1},u_{k-1})$, $k=j+1,\hdots,n-1$. Hence, $V_j^{R^1}(x)\leq V_j^{R^2}(x)$ and thus, $V_j^{R}(x)$ is a non-decreasing function of $R\in \mathbb{R}^{+}$.\\
 Next, for a fixed $\pi\in \Pi^{DM}_{j,n-1}$ consider two points $(R^1,V_j^{\pi,R^1})$, $(R^2,V_j^{\pi,R^2})$ such that $\{Q^1_k(\cdot|x_{k-1},u_{k-1}):k=j+1,\hdots,n\}$ achieves the supremum in (\ref{eq4.2300aamcm1}) for $R^1$, and $\{Q^2_k(\cdot|x_{k-1},u_{k-1}):k=j+1,\hdots,n\}$ achieves the supremum in (\ref{eq4.2300aamcm1}) for $R^2$. Then
 \begin{eqnarray*} 
 ||Q_k^1(\cdot|x_{k-1},u_{k-1})-Q^o_k(\cdot|x_{k-1},u_{k-1})||_{TV}&\leq & R^1, \quad k=j+1,\hdots,n-1\\
 ||Q_k^2(\cdot|x_{k-1},u_{k-1})-Q^o_k(\cdot|x_{k-1},u_{k-1})||_{TV}&\leq &R^2, \quad k=j+1,\hdots,n-1.
 \end{eqnarray*}
For any $\lambda\in(0,1)$ we have
\begin{align}
  &||\lambda Q^1_k(\cdot|x_{k-1},u_{k-1})+(1-\lambda)Q^2_k(\cdot|x_{k-1},u_{k-1})-Q^o_k(\cdot|x_{k-1},u_{k-1})||_{TV}\nonumber\\
  &\leq \lambda ||Q^1_k(\cdot|x_{k-1},u_{k-1})-Q^o_k(\cdot|x_{k-1},u_{k-1})||_{TV}+(1-\lambda)||Q^2_k(\cdot|x_{k-1},u_{k-1})\label{concav}\\
  &  -Q^o_k(\cdot|x_{k-1},u_{k-1})||_{TV}\leq \lambda R^1+(1-\lambda)R^2,\hst k=j+1,\hdots,n.\nonumber
\end{align}
 Define $Q^*_k(\cdot|x_{k-1},u_{k-1})\triangleq \lambda Q^1_k(\cdot|x_{k-1},u_{k-1})+(1-\lambda)Q^2_k(\cdot|x_{k-1},u_{k-1})$, $R=\lambda R^1+(1-\lambda)R^2$. By (\ref{concav}), $Q_k^*\in \mathbf{B}_{R}(Q^o_k)(x_{k-1},u_{k-1})$, $k=j+1,\hdots,n$. Define the unique probability measure \begin{align*}
  Q^*_{j+1,n}(dx^n||u^n){\triangleq} \lambda \otimes_{k=j+1}^nQ^1_k(dx_k|x_{k-1},u_{k-1}){+}(1{-}\lambda)\otimes_{k=j+1}^n Q^2_k(dx_k|x_{k-1},u_{k-1}).
\end{align*}
 Then, \begin{align*}
  V_j^{\pi,R}(x)\geq \int \Big(\sum_{k=j}^{n-1}f_k(x_k,u_k)+h_n(x_n)\Big)Q^*_{j+1,n}(dx^n||u^n)).
\end{align*}
 Hence, \begin{eqnarray*}
  V_j^{\pi,R}(x_j)&=&\mbox{RHS of (\ref{eq4.2300aamcm1})}\\
  &\geq & \lambda\int \Big(\sum_{k=j}^{n-1}f_k(x_k,u_k)+h_n(x_n)\Big)\otimes_{k=j+1}^nQ_k^1(dx_k|x_{k-1},u_{k-1})\\
  &&+(1-\lambda)\int \Big(\sum_{k=j}^{n-1}f_k(x_k,u_k)+h_n(x_n)\Big)\otimes_{k=j+1}^nQ_k^2(dx_k|x_{k-1},u_{k-1})\\
  &=&\lambda V_j^{\pi,R^1}(x_j)+(1-\lambda)V_j^{\pi,R^2}(x_j),\quad\mbox{$ j=0,\hdots,n-1$}.
\end{eqnarray*}
 Hence, for any $\pi\in\mathbf{\Pi}^{DM}_{j,n-1}$, $V_j^{\pi,R}(x_j)$ is a concave function of $R$, and thus it is also concave for the $\pi\in\mathbf{\Pi}^{DM}_{j,n-1}$, which achieve the infimum in (\ref{vf1a}).\hfill
\end{proof}

This concavity property of the pay-off is also verified in the examples presented in Section \ref{sec.examples}.
\begin{remark}
The previous results apply to randomized strategies as well.
\end{remark}

\subsection{Minimax Dynamic Programming for Infinite Horizon D-MCM Subject to Ambiguity}
\label{sec.dpihc}
 In this section, we consider the infinite horizon version of the finite horizon D-MCM, and we derive similar results. In addition, we show that the operator associated with the dynamic programming equation is contractive, and we introduce a new policy iteration algorithm.
 
Consider the problem of minimizing the finite horizon cost
\begin{align}
  \sup_{       {\sr  {{ Q}_{k}(\cdot| x,u)  \in \mathbf{B}_{R_k}(Q^o_k)(x,u)}{ k=0,1,\ldots, n}}   }\mathbb{E}_{\mathbf{Q}^\pi_{\nu}}\Big\{\sum_{j=0}^{n-1}\alpha^jf(x_j^g,u_j^g)\Big\}\label{minfh1}
\end{align}
 with $0<\alpha<1$. By Theorem \ref{th4.111} the value function of \eqref{minfh1}, denoted by $V_j(x)$, $j=0,\hdots,n$, $x\in {\cal X}_{j}$ satisfies the dynamic programming equations (\ref{eq4.2058}), (\ref{eq4.2057}) with $h_n=0$, $R_j=R$, ${\cal X}_j={\cal X}$, ${\cal U}_j={\cal U}$, ${\cal U}_j(x)={\cal U}(x)$ and $Q^o_j(\cdot|\cdot)=Q^o(\cdot|\cdot)$ . Define $v_i(x)=\alpha^{i-n}V_{n-i}(x)$, where $0\leq i\leq n$ is the time to go, (see \cite{vanSchuppen10}). Then,
 \begin{align}
  &v_0(x){=}0\label{infvaltimevar1}\\
  &v_i(x){=}\smashoperator{\inf_{u\in {\cal U}(x)}}\hspace{.12cm} \Big\{f(x,u){+}\alpha\int_{\mathrlap{{\cal X}}}v_{i-1}(z)Q^o(dz|x,u){+}\alpha\frac{R}{2}\big(\smashoperator{\sup_{z\in{\cal X}}}v_{i-1}(z){-}\smashoperator{\inf_{z\in{\cal X}}}v_{i-1}(z)\big)\Big\}.\label{infvaltimevar2}
    \end{align}
In contrast with finite horizon case the one given by (\ref{infvaltimevar1})-(\ref{infvaltimevar2}) proceeds from lower to higher values of indices $i$. The dynamic programming for the discounted cost \bea\mathbb{E}_{\mathbf{Q}^\pi_{\nu}}\Big\{\sum_{j=0}^{\infty}\alpha^jf(x_j^g,u_j^g)\Big\}\eea is given by 
\begin{equation}\label{eq22.70}
  v_{\infty}(x)=\smashoperator{\inf_{u\in {\cal U}(x)}}\hspace{.12cm}\Big\{f(x,u)+\alpha\int_{\mathrlap{{\cal X}}}v_{\infty}(z)Q^o(dz|x,u)+\alpha\frac{R}{2}\Big(\sup_{z\in{\cal X}}v_{\infty}(z)-\inf_{z\in{\cal X}}v_{\infty}(z)\Big)\Big\}.
\end{equation}
 The maximizing conditional distribution is 
 \begin{eqnarray}
 &&Q^*\left({\cal X}^+|x,u\right)=Q^o({\cal X}^+| x,u) + \frac{R}{2} \in[0,1],\quad\mbox{$(x,u)\in \mathbb{K}$}\label{infpp1a}\\
 &&Q^*\left({\cal X}^-|x,u\right)=Q^o({\cal X}^-| x, u) - \frac{R}{2}\in[0,1],\quad\mbox{ $(x,u)\in \mathbb{K}$}\label{infppc1}\\
 &&Q^*\left(A|x,u\right)=Q^o(A| x, u),\quad\mbox{ $\forall A\subseteq  {\cal X}\setminus {\cal X}^+\cup{\cal X}^-$},\quad\mbox{$ (x,u)\in \mathbb{K}$}\label{infppc2}
\end{eqnarray}
where
 \begin{eqnarray}
&&{\cal X}^{+} \triangleq  \Big\{ x {\in} {\cal X}:     V(x)= \sup \{ V(x): x {\in} {\cal X} \}  \Big\}\label{infpp11aa}\\
&&{\cal X}^{-}  \triangleq  \Big\{ x {\in} {\cal X}:   V(x)=\inf \{ V(x): x {\in} {\cal X} \} \Big\}.\label{infppcc22}
\end{eqnarray}

Next, we show that the operator in the right hand side of (\ref{eq22.70}) is contractive.\\
\begin{lemma}\label{contractionmappinglemma}
   Let $L$ be the class of all measurable functions $V:{\cal X}\longrightarrow \mathbb{R}$, with finite norm $||V||\triangleq\max_{x\in {\cal X}}|V(x)|$, and $T: L\longmapsto L$ defined by 
   \begin{equation}\label{contreq1}
    (TV)(x)=\inf_{u\in{\cal U}(x)}\Big\{f(x,u)+\alpha\int_{{\cal X}}V(z)Q^o(dz|x,u)+\alpha\frac{R}{2}\Big(\sup_{z\in {\cal X}}V(z)-\inf_{z\in{\cal X}}V(z)\Big)\Big\}.
  \end{equation} 
If $V\in BC^+(\cal{X})$ and $\sup_{z\in{\cal X}}V(z)$, $\inf_{z\in{\cal X}}V(z)$ are finite, then $T$ is a contraction.
\end{lemma}

\begin{proof} For $V_1,V_2\in L$,
\begin{align*}
&(TV_1)(x)-(TV_2)(x)=\\
&\inf_{u\in {\cal U}(x)}\Big\{f(x,u)+\alpha\int_{\cal X}V_1(z)Q^o(dz|x,u)+\alpha\frac{R}{2}\Big(\sup_{z\in{\cal X}}V_1(z)-\inf_{z\in{\cal X}}V_1(z)\Big)\Big\}\\
&-\inf_{u\in {\cal U}(x)}\Big\{f(x,u)+\alpha\int_{\cal X}V_2(z)Q^o(dz|x,u)+\alpha\frac{R}{2}\Big(\sup_{z\in{\cal X}}V_2(z)-\inf_{z\in{\cal X}}V_2(z)\Big)\Big\}.
\end{align*}
Let 
\begin{equation*}
v\triangleq \arg\inf_{u\in {\cal U}(x)}\Big\{f(x,u)+\alpha\int_{\cal X}V_2(z)Q^o(dz|x,u)  +\alpha\frac{R}{2}\Big(\sup_{z\in{\cal X}}V_2(z)-\inf_{z\in{\cal X}}V_2(z)\Big)\Big\}.
\end{equation*}  
Then,
\begin{align*}
  &(TV_1)(x)-(TV_2)(x)\nonumber\\
  &=\inf_{u\in {\cal U}(x)}\Big\{f(x,u)+\alpha\int_{\cal X}V_1(z)Q^o(dz|x,u)+\alpha\frac{R}{2}\Big(\sup_{z\in{\cal X}}V_1(z)-\inf_{z\in{\cal X}}V_1(z)\Big)\Big\}\nonumber\\
  &\hphantom{=\inf_{u\in {\cal U}(x)}\Big\{f(x,}-\Big\{f(x,v)+\alpha\int_{\cal X}V_2(z)Q^o(dz|x,v)+\alpha\frac{R}{2}\Big(\sup_{z\in{\cal X}}V_2(z)-\inf_{z\in{\cal X}}V_2(z)\Big)\Big\}\nonumber\\
  &\leq \Big\{f(x,v)+\alpha\int_{\cal X}V_1(z)Q^o(dz|x,v)+\alpha\frac{R}{2}\Big(\sup_{z\in{\cal X}}V_1(z)-\inf_{z\in{\cal X}}V_1(z)\Big)\Big\}\nonumber\\
  &\hphantom{=\inf_{u\in {\cal U}(x)}\Big\{f(x,}-\Big\{f(x,v)+\alpha\int_{\cal X}V_2(z)Q^o(dz|x,v)+\alpha\frac{R}{2}\Big(\sup_{z\in{\cal X}}V_2(z)-\inf_{z\in{\cal X}}V_2(z)\Big)\Big\}\nonumber\\
  &\overset{(a)}=\Big\{\alpha\int_{\mathrlap{\cal X}}V_1(z)Q^{V_1}(dz|x,v) \Big\}{-}\Big\{\alpha\int_{\mathrlap{\cal X}}V_2(z)Q^o(dz|x,v)
  +\alpha\frac{R}{2}\Big(\sup_{z\in{\cal X}}V_2(z)-\inf_{z\in{\cal X}}V_2(z)\Big)\Big\}\nonumber\\
  &\overset{(b)}\leq\Big\{\alpha\int_{\cal X}V_1(z)Q^{V_1}(dz|x,v) \Big\}-\Big\{\alpha\int_{\cal X}V_2(z)Q^{V_1}(dz|x,v) \Big\}\nonumber\\
  &=\alpha\int_{\cal X}\left(V_1(z)-V_2(z)\right)Q^{V_1}(dz|x|v)
  \leq \alpha \sup_{z\in {\cal X}}|V_1(z)-V_2(z)|=\alpha||V_1-V_2||
\end{align*}
where (a) is obtained by applying (\ref{f2n}), with $\ell\equiv \alpha V_1$, $\nu^*(\cdot)\equiv Q^{V_1}(\cdot|\cdot)$, $\mu(\cdot)\equiv Q^o(\cdot|\cdot)$, and (b) is obtained by first apply (\ref{f2n}) as in (a) with $Q^{V_2}$ and then replace $Q^{V_2}$ by $Q^{V_1}$ which is suboptimal hence, the upper bound. By reversing the roles of $V_1$ and $V_2$ we get $(TV_2)(x)-(TV_1)(x)\leq \alpha ||V_2-V_1||$. Hence, $|(TV_1)(x)-(TV_2)(x)|\leq \alpha ||V_1-V_2||$ for all $x\in {\cal X}$, and
\begin{align*}
  ||TV_1-TV_2||\triangleq\max_{x\in {\cal X}}|(TV_1)(x)-(TV_2)(x)|\leq \alpha||V_1-V_2||
\end{align*}
 which implies that the operator $T:L\longmapsto L$ is a contraction.\hfill \end{proof}
\\

Utilizing Lemma \ref{contractionmappinglemma} we obtain the following theorem which is analogous to the classical result given in \cite{vanSchuppen10}.\\

\begin{theorem}\label{theoremcontr} Assume $v_{\infty}\in BC^+(\cal{X})$ and $\sup_{z\in{\cal X}}v_{\infty}(z)$, $\inf_{z\in{\cal X}}v_{\infty}(z)$ are finite.

(1) The dynamic programming equation 
\begin{equation*}
v_{\infty}(x)=\inf_{u\in {\cal U}(x)}\Big\{f(x,u)+\alpha\int_{{\cal X}}v_{\infty}(z)Q^o(dz|x,u)+\alpha\frac{R}{2}\Big(\sup_{z\in{\cal X}}v_{\infty}(z)-\inf_{z\in{\cal X}}v_{\infty}(z)\Big)\Big\}
\end{equation*} has a unique solution.

(2) Moreover, \begin{align*}
  v_{\infty}(x)=\inf_{g\in {\cal U}(x)} \mathbb{E}_{Q^*}\Big\{\sum_{j=0}^{\infty}\alpha^jf(x_j,u_j)|x_0=x\Big\}.
\end{align*}

(3) The mapping $T$ defined by \begin{align*}
  (TV)(x)=\inf_{u\in{\cal U}(x)}\Big\{f(x,u)+\alpha\int_{{\cal X}}V(z)Q^o(dz|x,u)+\alpha\frac{R}{2}\Big(\sup_{z\in {\cal X}}V(z)-\inf_{z\in{\cal X}}V(z)\Big)\Big\}
\end{align*} is a contraction mapping with respect to the norm $||V||=\max_{x\in {\cal X}}|V(x)|$.

(4) For any $V$, $\lim_{n\rightarrow \infty}||T^nV-v_{\infty}||=0$ and so \begin{align*}
  \lim_{n\longrightarrow \infty}(T^nV)(x)=v_{\infty}(x),\quad \mbox{for all}\ \ x\in {\cal X}.
\end{align*}
\end{theorem}

\begin{proof}
(1) Follows from \cite{vanSchuppen10} (Theorem $6.3.6$, part (a)).

(2) We need to show that $v_{\infty}(x)$ is the minimum value of $\mathbb{E}_{Q^*}\left\{\sum_{j=0}^{\infty}\alpha^jf(x_j,u_j)\right\}$ starting in state $x_0=x$. Recall that $0\leq f(x,u)\leq M$ for all $x\in{\cal X}$, $u\in{\cal U}(x)$. Clearly, with $x_0=x$ and for all $n$, \begin{align*}
  \inf_{g\in {\cal U}(x)} \mathbb{E}_{Q^*}\Big\{\sum_{j=0}^{\infty}\alpha^jf(x_j,u_j)\Big\}\geq \inf_{g\in {\cal U}(x)} \mathbb{E}_{Q^*}\Big\{\sum_{j=0}^{n-1}\alpha^jf(x_j,u_j)\Big\}=v_n(x).
\end{align*} Hence, $
   \inf\limits_{g\in {\cal U}(x)} \mathbb{E}_{Q^*}\Big\{\sum_{j=0}^{\infty}\alpha^jf(x_j,u_j)\Big\}\geq \lim\limits_{n\rightarrow \infty}v_n(x)=v_{\infty}(x)$. Conversely, for all $n$ \begin{align*}
  \inf_{g\in {\cal U}(x)} \mathbb{E}_{Q^*}\big\{\sum_{j=0}^{\infty}\alpha^jf(x_j,u_j)\big\}{\leq} \inf_{g\in {\cal U}(x)} \mathbb{E}_{Q^*}\big\{\sum_{j=0}^{n-1}\alpha^jf(x_j,u_j)\big\}{+}\sum_{j=n}^{\infty}\alpha^jM{=}v_n(x){+}\frac{\alpha^nM}{1-\alpha}
\end{align*} and so \begin{align*}
   \inf_{g\in {\cal U}(x)} \mathbb{E}_{Q^*}\Big\{\sum_{j=0}^{\infty}\alpha^jf(x_j,u_j)\Big\}\leq \lim_{n\rightarrow \infty}\Big[v_n(x)+\frac{\alpha^nM}{1-\alpha}\Big]=v_{\infty}(x).
\end{align*} Hence, $\inf\limits_{g\in {\cal U}(x)} \mathbb{E}_{Q^*}\Big\{\sum_{j=0}^{\infty}\alpha^jf(x_j,u_j)\Big\}=v_{\infty}(x)$.

(3) This follows from Lemma \ref{contractionmappinglemma}.

(4) Follows from \cite{vanSchuppen10} (Theorem $6.3.6$, part (b)).\hfill
\end{proof}

\subsubsection{Policy Iteration Algorithm}
Next, we present a modified version of the classical policy iteration algorithm \cite{varayia86}. From part 4 of Theorem \ref{th4.111}, the policy improvement and policy evaluation steps of a policy iteration algorithm must be performed using the maximizing conditional distribution obtained under total variation distance ambiguity constraint. Hence, in addition to the classical case, in which the policy improvement and evaluation steps are performed using the nominal conditional distribution, here, under the assumption that $f(\cdot)$ is bounded and non-negative, by invoking the results developed in earlier sections we propose a modified algorithm which is expected to converge to a stationary policy in a finite number of iterations, since both state space ${\cal X}$ and control space ${\cal U}$ are finite sets, and that at each iteration a better stationary policy will be obtained.

First, we introduce some notation. Since the state space ${\cal X}$ is a finite set, with say, $n$ elements, any function $V: {\cal X}\longrightarrow \mathbb{R}^n$ may be represented by vector in $\mathbb{R}^n$ defined by
\begin{align*}
 V(x)\triangleq\left(\begin{array}{ccc}
V(x_1) & \cdots & V(x_n)
\end{array} \right)^T\in \mathbb{R}^n.
\end{align*}
Write $z\leq y$, if $z(i)\leq y(i)$, for $\forall i\in \mathbb{Z}^n\triangleq \{1,2,\hdots,n\}$; and $z<y$ if $z\leq y$ and $z\neq y$. For a stationary control law $g$, let \begin{align*}
  f(g)=\left(\begin{array}{ccc}
               f(x_1,g(x_1)) &\cdots &f(x_n,g(x_n))
             \end{array}
  \right)^T
\end{align*}
 and define each entry of the transition matrix $Q^o(g){\in}\mathbb{R}^{n\times n}$ by $Q^o_{ij}(g)=Q^o(x_j|x_i,g(x_i))\equiv Q^{g,o}(x_i|x_j)$. Rewrite (\ref{contreq1}) (with $\sup_{z\in {\cal X}}V(z)$ denoting componentwise supremum, and similarly for the infimum) as \begin{align*}
  TV=\min_{g\in \mathbb{R}^n}\Big\{f(g)+\alpha Q^o(g)V+\alpha\frac{R}{2}\Big\{\sup_{z\in {\cal X}}V(z)-\inf_{z\in{\cal X}}V(z)\Big\}\Big\}
\end{align*}
 which by Theorem \ref{th4.111} is equivalent to \begin{align*}
  TV=\min_{g\in \mathbb{R}^n}\Big\{f(g)+\alpha Q^*(g)V\Big\}
\end{align*}
 where $Q^*(g)\in \mathbb{R}^{n\times n}$ and is given by (\ref{infpp1a})-(\ref{infppc2}). Note that, the minimization is taken componentwise, i.e., $g(x_1)$ is the minimum of the first component of $f(g)+\alpha Q^*(g)V$ and so on. For each stationary policy $g$, define $T(g):\mathbb{R}^n\longrightarrow \mathbb{R}^n$ by \begin{align*}
  T(g)V=f(g)+\alpha Q^*(g)V.
\end{align*}
 Then, $T(g)$ is a contraction mapping on the space of bounded continuous functions to itself, and from Theorem \ref{theoremcontr} it follows that \begin{align*}
  V(g)=T(g)V=f(g)+\alpha Q^*(g)V
\end{align*}
 has a unique solution $V(g)\in \mathbb{R}^n$. Next, we give the policy iteration algorithm.\\

\begin{algorithm}[Policy Iteration] Consider the notation above.\\

Initialization. Let $m=0$. Solve the equation \begin{align*}
  f(g_0)+\alpha Q^o(g_0)V_{Q^o}(g_0)=V_{Q^o}(g_0)\ \ \mbox{for}\ \ V_{Q^o}(g_0)\in \mathbb{R}^n.
\end{align*}
Identify the support sets using (\ref{infpp11aa})-(\ref{infppcc22}) and the analogue of $\Sigma_k$ of Section \ref{abs}, and construct the matrix $Q^*(g_0)$ using (\ref{infpp1a})-(\ref{infppc2}). Solve the equation \begin{align*}
  f(g_0)+\alpha Q^*(g_0)V_{Q^*}(g_0)=V_{Q^*}(g_0)\ \ \mbox{for}\ \ V_{Q^*}(g_0)\in \mathbb{R}^n.
\end{align*}
 
 $1$. For $m=m+1$ while $\min\limits_{g\in \mathbb{R}^n}\Big\{f(g)+\alpha Q^*(g)V_{Q^*}(g_{m-1})\Big\}<V_{Q^*}(g_{m-1})$ do:
\bi
\item[{(a)}] (Policy Improvement) Let $g_m\in \mathbb{R}^n$ be such that \begin{align*}
  f(g_m)+\alpha Q^*(g_m)V_{Q^*}(g_{m-1})=\min_{g\in \mathbb{R}^n}\Big\{f(g)+\alpha Q^*(g)V_{Q^*}(g_{m-1})\Big\}.
\end{align*}
\item[{(b)}] (Policy Evaluation) Solve the following equation for $V_{Q^o}(g_m)\in\mathbb{R}^n$ \begin{align*}
  f(g_m)+\alpha Q^o(g_m)V_{Q^o}(g_{m})=V_{Q^o}(g_{m}).
\end{align*}
 Identify the support sets using (\ref{pp11aa})-(\ref{ppcc22}), and construct the matrix $Q^*(g_m)$ using (\ref{pp1a})-(\ref{ppc2}). Solve the equation \begin{align*}
  f(g_m)+\alpha Q^*(g_m)V_{Q^*}(g_m)=V_{Q^*}(g_m)\ \ \mbox{for} \ \ V_{Q^*}(g_m)\in \mathbb{R}^n.
\end{align*}
\ei

$2$. Set $g^*=g_m$.
\end{algorithm}

In the next section, we illustrate through examples how the theoretical results obtained in preceding sections are applied. 

\section{Examples}
\label{sec.examples}
In Section \ref{working example} we illustrate an application of the finite horizon minimax problem to the well-known machine replacement example, and in Section \ref{working example2} we illustrate an application of the infinite horizon minimax problem for discounted cost by employing the policy iteration algorithm.

\subsection{Finite Horizon MCM}
\label{working example}
Consider a machine replacement example inspired by \cite{bertsekas05}. Specifically, we have a machine that is either running or is broken down. If it runs throughout one week, it makes a profit of \EUR{100} for that week. If it fails during the week, the profit is zero for that week. If it is running at the start of the week and we perform preventive maintenance, the probability that it will fail during the week is $0.4$. If we do not perform such maintenance, the probability of failure is $0.7$. The maintenance cost is set at \EUR{20}. When the machine is broken down at the start of the week, it may either be repaired at a cost of \EUR{40}, in which case it will fail during the week with a probability of $0.4$, or it may be replaced at a cost of \EUR{150} by a new machine that is guaranteed to run through its first week of operation. Assume that after $N{>}1$ weeks the machine, irrespective of its state, is scrapped with no cost.

The system dynamics is of the form $ x_{k+1}=f_k(x_k,u_k,w_k)$, $k=0,1,\dots, N-1$, where the state $x_k$ is an element of a space $S_k=\{\textbf{R},\textbf{B}\}$, $\textbf{R}=\mbox{machine running}$, $\textbf{B}=\mbox{machine broken}$, the control $u_k$ is an element of a space $U_k(x_k)$, $U_k(\textbf{R})=\{m,nm\}$, $m=\mbox{maintenance}$, $nm=\mbox{no maintenance}$, $U_k(\textbf{B})=\{r,s\}$, $r=\mbox{repair}$, $s=\mbox{replace}$. The random disturbance has a nominal conditional distribution $w_k\sim \mu(\cdot|x_k,u_k)$.

Such a system can be described in terms of the discrete-time system equation $x_{k+1}=w_k$, where the nominal probability distribution of $w_k$ is given by
\begin{align*}
&\mu(w_k=\textbf{R}|x_k=\textbf{R},u_k=m)=0.6, & &\mu(w_k=\textbf{B}|x_k=\textbf{R},u_k=m)=0.4,\\
&\mu(w_k=\textbf{R}|x_k=\textbf{R},u_k=nm)=0.3,& &\mu(w_k=\textbf{B}|x_k=\textbf{R},u_k=nm)=0.7,\\
&\mu(w_k=\textbf{R}|x_k=\textbf{B},u_k=r)=0.6,& &\mu(w_k=\textbf{B}|x_k=\textbf{B},u_k=r)=0.4,\\
&\mu(w_k=\textbf{R}|x_k=\textbf{B},u_k=s)=1,& &\mu(w_k=\textbf{B}|x_k=\textbf{B},u_k=s)=0
\end{align*} 
and the input costs $C_u$ are given by: if $u=m$ then $C_m=\mbox{\EUR{20}}$, if $u=nm$ then $C_{nm}=\mbox{\EUR{0}}$, if $u=r$ then $C_r=\mbox{\EUR{40}}$, and if $u=s$ then $C_s=\mbox{\EUR{150}}$.  The cost per stage is $g_k(x_k,u_k,w_k)=C_{u_k}$ if $w_k=\textbf{R}$, and $g_k(x_k,u_k,w_k)=C_{u_k}+100$ if $w_k=\textbf{B}$.
Since it is assumed that after $N$ weeks the machine, irrespective of its state, is scrapped without incurring any cost the terminal cost is $ g_N(\textbf{R})=g_N(\textbf{B})=0$.

The dynamic programming algorithm for the minimax problem subject to total variation distance uncertainty is given by
\begin{eqnarray}
 V_N(x_N)&=&0\\
 \label{ex.fin.max}
V_k(x_k)&=&\min_{u_k\in U_k(x_k)}\max_{\nu(dw_k|x_k,u_k):||\nu(\cdot|x_k,u_k)-\mu(\cdot|x_k,u_k)||_{TV}\leq R}\Big\{\nonumber\\
&&\mathbb{E}\Big\{g_k(x_k,u_k,w_k)+V_{k+1}(f(x_k,u_k,w_k))\Big\}\Big\}\\
&=&\min_{u_k\in U_k(x_k)}\max_{\nu(dw_k|x_k,u_k):||\nu(\cdot|x_k,u_k)-\mu(\cdot|x_k,u_k)||_{TV}\leq R}\mathbb{E}\Big\{\ell_k(x_k,u_k,w_k)\Big\}\nonumber
\end{eqnarray}
where $ \ell_k(x_k,u_k,w_k)=g_k(x_k,u_k,w_k)+V_{k+1}(w_k)$, $k=0,1,\dots,N-1$.
To adress the maximization problem in \eqref{ex.fin.max}, for each $k=0,1,\dots,N-1$, $x_k\in\{\textbf{R},\textbf{B}\}$ and $u_k \in\{m,nm,r,s\}$, define the maximum and minimum values of $\ell(x_k,u_k,w_k)$ by
\begin{align*}
\ell_{\max}(x_k,u_k)\triangleq \max_{w_k\in\{\textbf{R},\textbf{B}\}}\ell(x_k,u_k,w_k),& &
\ell_{\min}(x_k,u_k)\triangleq \min_{w_k\in\{\textbf{R},\textbf{B}\}}\ell(x_k,u_k,w_k)
\end{align*}
and its corresponding support sets by
$\Sigma^0{=}\{w_k{\in}\{\textbf{R},\textbf{B}\}{:}\ell(x_k,u_k,w_k){=}\ell_{\max}(x_k,u_k)\}$, and $\Sigma_0{=}\{w_k{\in}\{\textbf{R},\textbf{B}\}{:}\ell(x_k,u_k,w_k){=}\ell_{\min}(x_k,u_k)\}$. By employing \eqref{all3}, the maximizing conditional probability distribution of the random parameter $w_k$ is given by 
\begin{subequations}\label{equation5.5}
\begin{align}
&\alpha=\min\Big(\frac{R}{2},1-\mu(\Sigma^0|x_k,u_k)\Big)\\
& \nu^*(\Sigma^0|x_k,u_k)=\mu(\Sigma^0|x_k,u_k)+\alpha,\qquad  \nu^*(\Sigma_0|x_k,u_k)=\Big(\mu(\Sigma_0|x_k,u_k)-\alpha\Big)^+.\label{equation5.56}
\end{align}
\end{subequations}
Based on this formulation, the dynamic programming equation is given by 
\begin{align}
 V_N(x_N)&=0\\
V_k(x_k)&=\min_{u_k\in U_k(x_k)}\mathbb{E}_{\nu^*(\cdot|\cdot,\cdot)}\Big\{g_k(x_k,u_k,w_k)+V_{k+1}(f(x_k,u_k,w_k))\Big\}.
\end{align}

We assume that the planning horizon is $N=3$. The optimal cost-to-go and the optimal control policy, for each week and each possible state, as a function of $R\in [0,2]$ are illustrated in Figure \ref{fig1}. Clearly, Figure \ref{fig1.2} depicts that the optimal cost-to-go is a non-decreasing concave function of $R$ as stated in Lemma \ref{lemmadecrconc}.

In addition, the optimum solution for two possible values of $R$ and for each week results in optimal control policies as depicted in Table \ref{table1}. By setting $R{=}0$, we choose to calculate the optimal control policy when the true conditional probability $\nu(\cdot|x_k,u_k)=\mu(\cdot|x_k,u_k),\ k{=}0,1,2$. This corresponds to the classical dynamic programming algorithm. By setting $R{=}0.85$, we choose to calculate the optimal control policy when the true conditional distribution $\nu(\cdot|x_k,u_k)\neq\mu(\cdot|x_k,u_k),\ k{=}0,1,2$. Taking into consideration the maximization (that is, by setting $R{>}0$) the dynamic programming algorithm results in optimal control policies which are more robust with respect to uncertainty, but with the sacrifice of low present and future costs. In cases in which we need to balance the desire for low costs with the undesirability of scenarios with high uncertainty, we must choose the appropriate value of $R$ by using Figure \ref{fig1.2}.

\begin{figure}[!htbp]
\centering
\subfloat[][]{
\label{fig1.2} 
\includegraphics[width=.8\linewidth]{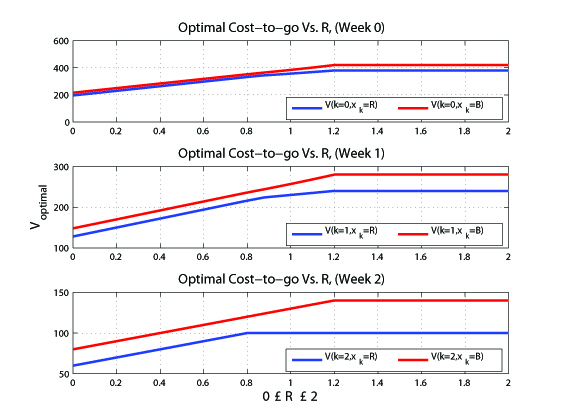}}\\
\subfloat[][]{
\label{fig1.3} 
\includegraphics[width=.8\linewidth]{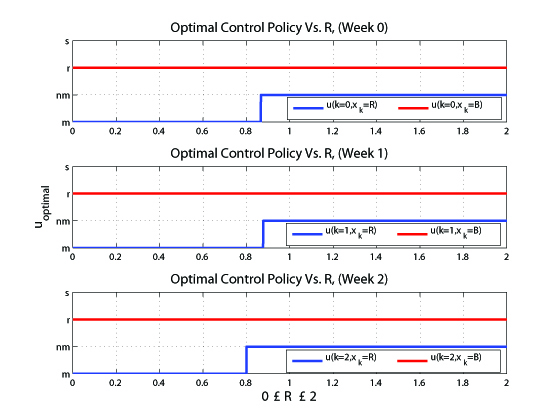}}
\caption[Optimal Solution of Example AA]{(a) Optimal Cost-to-Go; (b) Optimal Control Policy (``m"= maintenance, ``nm= no maintenace", ``r=repair", ``s=replace"). }
\label{fig1}
\end{figure}

\begin{table}[!ht]\setlength{\tabcolsep}{1pt}\centering
\begin{tabular}{cc|c|c||c|c||c|c|}
  \cline{3-8}
& \multirow{3}{*}{Stock}&\multicolumn{2}{ ||c|| }{Week.$0$}&\multicolumn{2}{ |c|| }{Week.$1$}&\multicolumn{2}{ |c| }{Week.$2$}\\ \cline{2-8}
\multicolumn{1}{ c| }{}
& & \multicolumn{1}{ ||c| }{Cost-to-go}&Optimal& Cost-to-go&Optimal& Cost-to-go&Optimal\\
\multicolumn{1}{ c| }{}
& &\multicolumn{1}{ ||c| }{} &  Policy& &  Policy& &  Policy\\  \cline{1-8}
\multicolumn{1}{ |c| }{\multirow{2}{*}{{$R=0$}}}&\textbf{R}& \multicolumn{1}{ ||c| }{196} & m     & 128  & m     & 60  & m\\
\multicolumn{1}{ |c| }{}
&\textbf{B} & \multicolumn{1}{ ||c| }{216} & r & 148  & r     & 80  & r\\ \hline 
\multicolumn{1}{ |c| }{\multirow{2}{*}{{$R=0.85$}}}&\textbf{R} & \multicolumn{1}{ ||c| }{340} & m     & 221  & m     & 100  & nm\\
\multicolumn{1}{ |c| }{}
& \textbf{B} & \multicolumn{1}{ ||c| }{360} & r     & 241  & r     & 122  & r\\ \hline
\end{tabular}
\caption{Dynamic Programming Algorithm Results}\label{table1}
\end{table}

\subsection{Infinite Horizon D-MCM}
\label{working example2}
 Here, we illustrate an application of the infinite horizon minimax problem for discounted cost, by considering the stochastic control system shown in Figure \ref{MCfigure}, with state space ${\cal X}=\{1,2,3\}$ and control set ${\cal U}=\{u_1,u_2\}$. 

Assume the nominal transition probabilities are given under controls $u_1$ and $u_2$ by
\begin{align}
  &Q^o(u_1)=\frac{1}{9}\left(
                    \begin{array}{ccc}
                      3 & 1 & 5 \\
                      4 & 2 & 3 \\
                      1 & 6 & 2 \\
                    \end{array}
                  \right),\hst Q^o(u_2)=\frac{1}{9}\left(
                    \begin{array}{ccc}
                      1 & 2 & 6 \\
                      4 & 2 & 3 \\
                      4 & 1 & 4 \\
                    \end{array}
                  \right)
\end{align}
 the discount factor is $\alpha=0.9$, the total variation distance radius is $R=\frac{6}{9}$, and the cost function under each state and action is \begin{align*}
  f(1,u_1)=2,\ f(2,u_1)=1,\ f(3,u_1)=3,\ f(1,u_2)=0.5,\ f(2,u_2)=3,\ f(3,u_2)=0.
\end{align*}

 Using policy iteration of Section \ref{dp}, with initial policies $g_0(1)=u_1$, $g_0(2)=u_2$, $g_0(3)=u_2$, the algorithm converge to the following optimal policy and value after two iterations.
\begin{align*}
g^*=g_2\triangleq \begin{pmatrix}
                   g_2(1) \\
                   g_2(2) \\
                   g_2(3)
                 \end{pmatrix}
=\begin{pmatrix}
                   u_2 \\
                   u_1 \\
                   u_2
                 \end{pmatrix}
,\hst V_{Q^*}(g^*)=V_{Q^*}(g_2)\triangleq \begin{pmatrix}
                   V_{Q^*}(1) \\
                   V_{Q^*}(2) \\
                   V_{Q^*}(3)
                 \end{pmatrix}
=\begin{pmatrix}
                   6.79 \\
                   7.43\\
                   6.32
                 \end{pmatrix}.
\end{align*}
 Figure \ref{fig2} depicts the optimal value functions for all possible values of $R$, and shows that, the value functions are non-decreasing and concave functions of $R$ as stated in Lemma \ref{lemmadecrconc}.
\begin{figure}[!ht]
\centering
\subfloat[]{
\label{MCfigure} 
\includegraphics[ width=.5\textwidth]{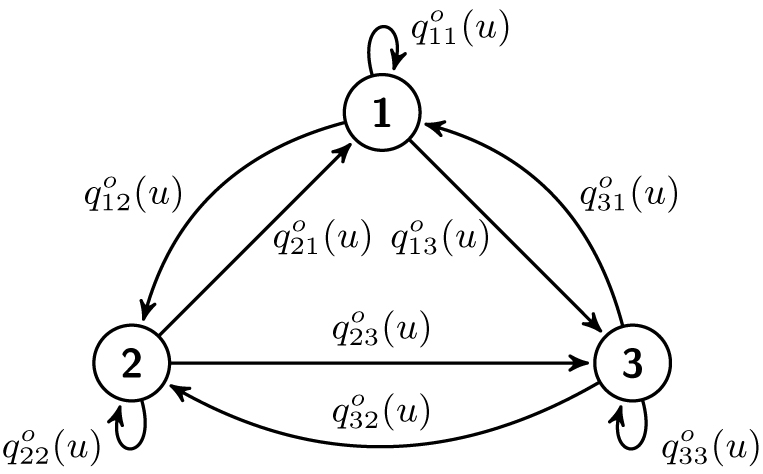}}
\subfloat[]{
\label{fig2} 
\includegraphics[width=.5\textwidth]{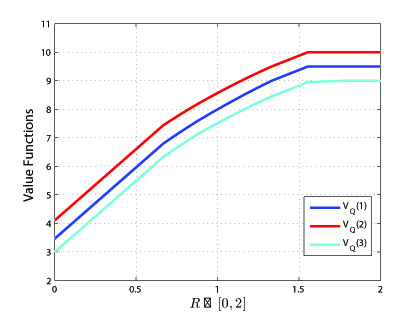}}
\caption[Optimal Solution of Example B]{(a) Transition Probability Graph; (b) Optimal Value as a Function of Total Variation Parameter.  }
\label{fig111}
\end{figure}

\section{Conclusions}

In this paper, we examined the optimality of stochastic control strategies via dynamic programming, when the ambiguity class is described by the total variation distance between the conditional distribution of the controlled process and the nominal conditional distribution. The problem is formulated using minimax strategies in which the control process seeks to minimize the pay-off while the controlled process seeks to maximize it over the total variation ambiguity class. By using concepts from signed measures a closed form expression of the maximizing measure is derived. It is then employed to obtain a new dynamic programming recursion which, in addition to the standard terms, includes the oscillator seminorm of the value function, while for the infinite horizon case a new discounted dynamic programming equation is obtained. It is shown that the dynamic programming operator is contractive, and a new policy iteration algorithm is developed for computing the optimal stochastic control strategies. Finally, we illustrate through examples the applications of our results.

\bibliographystyle{siam.bst}
\bibliography{DynamicProgramming}
\end{document}